\documentclass[11pt]{article}

\usepackage[table]{xcolor}
\usepackage{IEEEtrantools}
\usepackage{latexsym}
\usepackage{amsfonts}
\usepackage{amssymb}
\usepackage{amsmath,amsfonts,amsthm}
\usepackage{mathrsfs}
\usepackage[all]{xy}
\usepackage{epic}
\usepackage{eepic}
\usepackage{enumerate}
\usepackage{multirow}
\usepackage[breaklinks]{hyperref}
\usepackage{tikz}

\usepackage{mathrsfs}
\usepackage[all]{xy}
\usepackage{epic}
\usepackage{eepic}
\usepackage{enumerate}
\usepackage{multirow}
\usepackage[breaklinks]{hyperref}
\usepackage{tikz}
\usetikzlibrary{matrix,arrows,decorations.pathmorphing}
\usepackage{xcolor}
\usepackage{color}
\usepackage[breaklinks]{hyperref}
\usepackage{tikz-cd}

\usepackage{mathrsfs}
\usepackage[all]{xy}
\usepackage{epic}
\usepackage{eepic}
\usepackage{enumerate}
\usepackage{multirow}
\usepackage[breaklinks]{hyperref}
\usepackage{tikz}
\usetikzlibrary{arrows,decorations.markings}
\usepackage{color}
\usepackage{bbm}
\usepackage[most]{tcolorbox}
\usepackage{subcaption}
\usepackage{upgreek}
\usetikzlibrary{patterns}

\newcommand{\scrt}{\mathscr{T}}

\newcommand{\be}{\begin{equation}}
	\newcommand{\ee}{\end{equation}}
\newcommand{\bea}{\begin{eqnarray}}
	\newcommand{\eea}{\end{eqnarray}}
\newcommand{\bean}{\begin{eqnarray*}}
	\newcommand{\eean}{\end{eqnarray*}}
\newcommand{\brray}{\begin{array}}
	\newcommand{\erray}{\end{array}}
\newcommand{\biearray}{\begin{IEEEarray}{rCl}}
	\newcommand{\eiearray}{\end{IEEEarray}}

\newtheorem{dfn}{Definition}[section]
\newtheorem{thm}[dfn]{Theorem}
\newtheorem{lmma}[dfn]{Lemma}
\newtheorem{ppsn}[dfn]{Proposition}
\newtheorem{crlre}[dfn]{Corollary}
\newtheorem{xmpl}[dfn]{Example}
\newtheorem{rmrk}[dfn]{Remark}
\newtheoremstyle{case}{}{}{}{}{}{:}{ }{}
\theoremstyle{case}

\newtheoremstyle{claim}{}{}{}{}{}{:}{ }{}
\theoremstyle{claim}

\newcommand{\bdfn}{\begin{dfn}\rm}
	\newcommand{\bthm}{\begin{thm}}
		\newcommand{\blmma}{\begin{lmma}}
			\newcommand{\bppsn}{\begin{ppsn}}
				\newcommand{\bcrlre}{\begin{crlre}}
					\newcommand{\bxmpl}{\begin{xmpl}}
						\newcommand{\brmrk}{\begin{rmrk}\rm}
							
							\newcommand{\edfn}{\end{dfn}}
						\newcommand{\ethm}{\end{thm}}
					\newcommand{\elmma}{\end{lmma}}
				\newcommand{\eppsn}{\end{ppsn}}
			\newcommand{\ecrlre}{\end{crlre}}
		\newcommand{\exmpl}{\end{xmpl}}
	\newcommand{\ermrk}{\end{rmrk}}

\newcommand{\bbc}{\mathbb{C}}
\newcommand{\bbz}{\mathbb{Z}}

\newcommand{\prf}{\noindent{\it Proof\/}: }

\def \qed { \mbox{}\hfill
	$\Box$\vspace{1ex}}

\newcommand\restr[2]{{
		\left.\kern-\nulldelimiterspace 
		#1 
		\littletaller 
		\right|_{#2} 
}}

\newcommand{\littletaller}{\mathchoice{\vphantom{\big|}}{}{}{}}

\newenvironment{proof*}[1][\proofname]{\proof[#1]}{\endproof}
\begin{document}
	
	\author{{\sc Surajit Biswas and  Bipul Saurabh}}
	\date{}
	\title{Vanishing of dimensions and nonexistence of spectral triples on compact Vilenkin groups}
	\maketitle
	
		\begin{abstract} 
	We compute the spectral dimension, the dimension of a symmetric random walk, and the Gelfand-Kirillov dimension for compact Vilenkin groups. As a result, we show that these dimensions are zero for any compact, totally disconnected, metrizable topological group. We provide an explicit description of the $K$-groups for compact Vilenkin groups. We express the generators of the $K_0$-groups in terms of the corresponding matrix coefficients for two specific examples: the group of $p$-adic integers and the $p$-adic Heisenberg group. Finally, we prove the nonexistence of a natural class of spectral triples on the group of $p$-adic integers.
		\end{abstract}
		
{\bf AMS Subject Classification No.:} {\large 58}B{\large 32}, {\large 58}B{\large 34}, {\large 46}L{\large 80}.

\section{Introduction}

Compact abelian Vilenkin groups  were first studied in  \rm\cite{Vil-1947aa} to explore Fourier methods through their characters.  Among these, the most extensively examined group in noncommutative geometry is the additive group of $p$-adic integers. Compact Vilenkin groups (CVGs), being compact, metrizable, and totally disconnected, have  the Lebesgue covering dimension and the inductive dimensions equal to zero.  However, several other notions of dimension have emerged in the context of noncommutative geometry, prompting a natural question: do these alternative notions of dimension also vanish for CVGs? In this article, we investigate three such notions and answer this question in the affirmative. Chakraborty and Pal (\rm\cite{ChaPal-2016aa}) introduced an invariant called the spectral dimension for an ergodic $C^*$-dynamical systems. They further conjectured that the spectral dimension of a homogeneous space of a compact Lie group coincides with its dimension as a differentiable manifold. If we extend this to a totally disconnected compact group $G$, one can expect the spectral dimension to be zero. Here we prove  that if $G$ is in addition, metrizable or equivalently a CVG then this is indeed the case. The next dimension that we studied is the notion of the probability $p_k$ of returning at $1$ after $2k$ steps for a symmetric random walk on a finitely generated Hopf algebra introduced by  Banica and Vergnioux \rm\cite{BanVer-2009ab}. We  proved that if $G$ is a connected simply connected compact real Lie group with dimension $d$ then $p_k$ decays at the rate of  $k^{-d/2}$. It simply means that the quantity $\lim_{k \rightarrow \infty} \frac{-2\log p_k}{\log k}$ gives the dimension of $G$. Based on these observations, we define the dimension $d_{RW}^G$ of a symmetric random walk on a CQG. By leveraging the properties of the representation ring of a CVG $G$, we reduce the computation to that of a finite group and prove that $d_{RW}^G=0$. 

Our next objective is to investigate the noncommutative geometric aspects of these spaces along the line of Connes (\rm\cite{Con-1994aa}).  We first compute the 
$K$-groups of $C(G)$ and explicitly describe their generators. This is achieved by utilizing the fact that 
$C(G)$ can be realized as the inductive limit of a sequence of $C^*$-algebras associated with certain quotient spaces of 
$G$. To facilitate further analysis, we fix a prime $p$ and consider two specific examples: the group $\bbz_p$ of $p$-adic integers and the $p$-adic Heisenberg group $\mathbb{H}_d\left(\mathbb{Z}_p\right)$. We show that the generators of their $K_0$-groups can be expressed as finite linear combinations of matrix co-efficients of finite dimensional representation of the group.  This characterization is particularly useful when pairing elements of $K_0$-group  with an even spectral triple.  Finally, we present two negative results. First, we show that no nontrivial even spectral triple on $C(\mathbb{Z}_p)$ can be equivariant with respect to its co-multiplication map. Second, we prove that there exists no even spectral triple on the group of $p$-adic integers whose associated Dirac operator is a generalized weighted shift operator with nonzero index. To the best of our knowledge, these are among the first results of their kind: while spectral triples on $\mathbb{Z}_p$ have been studied before, the focus has typically been on their metric aspects, with little attention given to their $K$-theoretic nontriviality. Our results highlight key obstructions to constructing nontrivial spectral triples, and thus to developing Connes' framework of noncommutative geometry in the $p$-adic setting.

Let us now outline the organization of the paper. In the next section, we recall the definition of spectral dimension from \rm\cite{ChaPal-2016aa}, and the Gelfand Kirillov dimension from \rm\cite{GelKir-1966ab}.  Additionally, motivated by \rm\cite{BanVer-2009ab}, we introduce the notion of the dimension of a symmetric random walk on a compact quantum group $G$. In Section~\ref{Dimensional Invariants of Compact Vilenkin Groups}, we compute these dimensions for CVGs, showing that they all vanish. As a consequence, these dimensions are also zero for all compact, totally disconnected, and metrizable topological groups. Section~\ref{$K$-groups} deals  with the $K$-theory of CVGs. For a CVG $G$, we show that $K_1(C(G)) = 0$ and explicitly identify the generators of $K_0(C(G))$. We then focus on two examples: the (commutative) additive group of $p$-adic integers $\mathbb{Z}_p$ and the (non-commutative) $p$-adic Heisenberg group over $\mathbb{Z}_p$, describing the generators of their $K_0$-groups in terms of matrix coefficients.   Section \ref{Spectral triple} deals with  the non-existence of a class of spectral triples on the group of $p$-adic integers.

All Hilbert spaces and algebras are over the field $\bbc$. We denote the cardinality of a set $X$ by $|X|$. For a subset $B$ of a set $X$, we denote the characteristic function of $B$ on $X$ as $\mathbbm{1}_B$. The letter $p$ denotes a prime number.

\section{Preliminaries}

In this section, we recall some preliminary  related to various dimensional notions in the context of compact quantum groups. Let's start with the definition of an action of a CQG on a $C^*$-algebra. 

\begin{dfn}
A compact quantum group $G$ acts on a $C^*$-algebra $A$ if there is a $^*$-homomorphism $\tau: A\rightarrow A\otimes C(G)$ satisfying
\begin{enumerate}[(i)]
\item $(\tau\otimes id)\tau = (id\otimes \Delta)\tau$,
\item $\left\{(I\otimes g)\tau (f): f\in A, g\in C(G)\right\}$ densely spans $A\otimes C(G)$.
\end{enumerate}
We call  $A$  a homogeneous space of $G$ if the fixed point subalgebra $\{f\in A:\tau(f)=f\otimes I\}$ is $\mathbb{C}I$. In such a case, the action $\tau$ is called ergodic, and the triple $\left(A,G,\tau\right)$ is called an ergodic $C^*$-dynamical system.

A covariant representation $(\pi, u)$ of a $C^*$-dynamical system $\left(A,G,\tau\right)$ consists of a unital $^*$-representation $\pi : A\rightarrow\mathcal{L}(\mathcal{H})$, a unitary representation $u$ of $G$ on $\mathcal{H}$ (i.e. a unitary element of the multiplier algebra $M\left(\mathcal{K}(\mathcal{H})\otimes C(G)\right)$ with $(id\otimes\Delta)(u)=u_{12}u_{13}$), satisfying $$(\pi\otimes id)\tau(f) = u\left(\pi (f)\otimes I\right)u^* \mbox{  for all  } f\in A.$$
\end{dfn}

\begin{dfn}
For a $C^*$-dynamical system $(A,G,\tau)$, an operator $D$ acting on a Hilbert space $\mathcal{H}$ is equivariant with respect to a covariant representation $(\pi, u)$ of the  system if $D\otimes I$ commutes with $u$. If $(\pi,u)$ is a covariant representation of $(A,G,\tau)$ on a Hilbert space $\mathcal{H}$ and $(\mathcal{H},\pi,D)$ is a spectral triple for a dense $^*$-subalgebra $\mathcal{A}$ of $A$, then we say that $(\mathcal{H},\pi, D)$ is equivariant with respect to $(\pi,u)$ if $D$ is equivariant with respect to $(\pi,u)$.
\end{dfn}

\noindent A homogeneous space $A$ for a compact quantum group $G$ has an invariant state $\rho$ satisfying
$$(\rho\otimes id)\tau(f)=\rho (f)I,\,f\in A.$$
This invariant state $\rho$ is unique and relates to the Haar state $\mathfrak{h}$ on $G$ through the equality
$$(id\otimes \mathfrak{h})\tau(f)=\rho (f)I,\,f\in A.$$

Given an ergodic $C^*$-dynamical system $(A,G,\tau)$ with unique invariant state $\rho$, let $\left(\mathcal{H}_\rho,\pi_\rho,\Lambda_\rho\right)$ denote the GNS representation associated with $\rho$, i.e. $\mathcal{H}_\rho$ is a Hilbert space, $\Lambda_\rho: A\rightarrow\mathcal{H}_\rho$ is linear with $\Lambda_\rho (A)$ dense in $\mathcal{H}_\rho$, and $\left\langle\Lambda_\rho (f),\Lambda_\rho (g)\right\rangle=\rho(f^*g)$; and $\pi_\rho:A\rightarrow\mathcal{L}\left(\mathcal{H}_\rho\right)$ is the $^*$-representation of $A$ on $\mathcal{H}_\rho$ defined by $\pi_\rho (f)\Lambda_\rho (g)=\Lambda_\rho (fg)$. The action $\tau$ induces a unitary representation $u_\tau$ of $G$ on $\mathcal{H}_\rho$, making $\left(\pi_\rho, u_\tau\right)$ a covariant representation of the system $\left(A,G,\tau\right)$. Let $\mathcal{O}(G)$ be the dense $^*$-subalgebra of $C(G)$ generated by the matrix entries of irreducible unitary representations of $G$. We define $$\mathcal{A}=\left\{f\in A:\tau(f)\in A\otimes_{alg} \mathcal{O}(G)\right\};$$
by \rm\cite[Theorem 1.5]{Pod-1995aa}, $\mathcal{A}$ is a dense $^*$-subalgebra of $A$. Let $\mathcal{E}$ be the class of spectral triples for $\mathcal{A}$ equivariant with respect to the covariant representation $\left(\pi_\rho,u_\tau\right)$. We define the spectral dimension of the system $(A,G,\tau)$ as the quantity
$$\inf\left\{s>0:\exists\,D \text{ such that } \left(\mathcal{H}_\rho,\pi_\rho,D\right)\in\mathcal{E} \text{ and } D \text{ is } s\text{-summable}\right\}.$$
We will denote this number by $\mathcal{S}dim\left(A,G,\tau\right)$.

\begin{ppsn}
For a compact topological group $G$, in the ergodic $C^*$-dynamical system $\left(C(G),G,\Delta\right)$, we have $\mathcal{A}=\mathcal{O}(G)$.
\end{ppsn}

\begin{proof}
Let $f\in \mathcal{A}$ with $\Delta (f)=\sum_{i=1}^n f_i\otimes \chi_i$ for $f_i\in A$ and $\chi_i\in\mathcal{O}(G)$. Using the identity element $e$ of $G$, we find that $f(x)=\Delta (f)(e,x)=\sum_{i=1}^n f_i(e)\cdot \chi_i(x)$ for all $x\in G$. Thus, $a=\sum_{i=1}^n f_i(e)\cdot\chi_i\in\mathcal{O}(G)$.
\end{proof}

For a finitely generated Hopf algebra $(A,u)$, Banica and Vergnioux (\rm\cite{BanVer-2009ab}) defined the notion of  probability $p_k$ of returning at $1$ after $2k$ steps and then proved that for a connected simply connected compact real Lie group $G$, $p_k$ decays at the rate of $k^{-d/2}$, where $d$ is the real dimension of $G$.  Following this observation, we define a notion of dimension $\mathrm{d}_{RW}^G$ of a random walk on a CQG.  
\bdfn Let  $(G,\Delta)$ be a CQG with Haar state $\mathfrak{h}$. Consider the set $\mathcal{R}_{\mathrm{sym}}(G)$ consisting of finite-dimensional representations $\chi$ of $G$ such that $\chi=\overline{\chi}$. For $\chi \in \mathcal{R}_{\mathrm{sym}}(G)$ with $\mbox{dim}(\chi)=k$, define the probability of returning to $1$ in $2n$ steps for a random walk on the Hopf algebra generated by matrix co-efficients of $\chi$ as follows:
\[
p_n^{\chi}= \frac{1}{k^{2n}}\mathfrak{h}\left(\chi_{11}+\chi_{22}+ \cdots \chi_{kk}\right)^{2n}.
\] 
Then the dimension of a (symmetric) random walk on $G$ is defined as
\[
\mathrm{d}_{RW}^G=\sup_{\chi \in \mathcal{R}}\Big( \overline{\lim}_{n \to \infty} \frac{-2\log  p_n^{\chi}}{\log n}\Big).
\]
\edfn
\noindent Next, we recall the definition of the Gelfand-Kirillov dimension from \rm\cite{GelKir-1966ab}.

\begin{dfn}
Let $\mathcal{A}$ be a unital algebra over a field $\kappa$, and let $\mathcal{V}$ denote the set of all finite-dimensional subspaces of $\mathcal{A}$. For each $V \in \mathcal{V}$ and $n \in \mathbb{N}_0$, define $V_n$ as the subspace of elements expressible as sums of products of at most $n$ elements of $V$, i.e., $V_n = \sum_{k=0}^n V^k$. The Gelfand-Kirillov dimension of $\mathcal{A}$ is then defined by
\[
\mathrm{GKdim}(\mathcal{A}) = \sup_{V \in \mathcal{V}} \overline{\lim}_{n \to \infty} \frac{\log \dim(V_n)}{\log n}.
\]
\end{dfn}

\section{Dimensional Invariants of Compact Vilenkin Groups}\label{Dimensional Invariants of Compact Vilenkin Groups}

We begin by recalling the definition of CVG.

\begin{dfn}\label{CVG}{\rm\cite{DelRod-2022Pre}}
A topological group $G$ is called a compact Vilenkin group if it possesses a strictly decreasing sequence  $G:=G_0\supset G_1\supset G_2\supset \ldots$ of compact open normal subgroups of $G$ such that
\begin{enumerate}[(i)]
\item $2\leq \left|G_n/G_{n+1}\right| =:\kappa_n < \infty$ for all $n\in\mathbb{N}_0$,
\item $\langle G_n\rangle_{n\in\mathbb{N}_0}$ is a basis of neighbourhoods at the identity element $e$ of $G$, and $\bigcap_{n\in\mathbb{N}_0}G_n =\{e\}$.
\end{enumerate}
\end{dfn}

\begin{xmpl}
The balls centered at the origin of a local field, such as the compact group of $p$-adic integers, are abelian Vilenkin groups. Note that, for the additive group of $p$-adic integers $G=\mathbb{Z}_p$, the associated compact open normal subgroups can be choosen as $G_n:= p^n\mathbb{Z}_p$, $n\in\mathbb{N}_0$. The group of multiplicative units in $\mathbb{Z}_p$, i.e., $\mathbb{Z}_p^{\times}$ is also an abelian CVG. A very basic example of a noncommutative CVG is the $p$-adic Heisenberg group, which will be discussed in the next section.
\end{xmpl}

We observe that for a CVG $G$, and each $n \in \mathbb{N}$, the finite family of cosets $\{ x\bullet G_n : [x] \in G / G_n \}$ forms a disjoint open cover of $G$, implying that each $G_n$ is closed. Since  $\langle G_n\rangle_{n\in\mathbb{N}_0}$  forms a neighbourhood base at $e$, it follows that the only  connected subset of $G$ containing $e$ is the singleton $\{e\}$. This shows that   $G$ is totally disconnected as the left-mulplication by any element of $G$ is a homeomorphism. Furthermore, by \rm\cite[Theorem~8.3]{HewRos-1979aa}, CVGs are metrizable. The following proposition is a slight generalization of the result in  \rm\cite[Section 1]{Onn-1972aa}, which states that  every infinite, compact, metrizable, zero-dimensional commutative topological group is a CVG. 

\begin{ppsn}\label{Characterization of CVG}
An infinite, compact, totally disconnected, and metrizable topological group is a compact Vilenkin group.
\end{ppsn}

\begin{proof}
Let \( G \) be an infinite, compact, totally disconnected, and metrizable topological group. By \rm\cite[Theorem 7.7]{HewRos-1979aa}, the compact open subgroups of \( G \) form a basis of neighborhoods at the identity element \( e \). Since \( G \) is infinite and compact, each open normal subgroup has finite index in \( G \) and is therefore infinite. In particular, for any compact open subgroup \( U \) of \( G \), the conjugates \( xUx^{-1} \) for \( x \in G \) are finite in number, and their intersection, \( \bigcap_{x \in G} xUx^{-1} \), is both normal and compact open in \( G \). Thus, the collection
\[
\mathcal{V} := \left\{ \bigcap_{x \in G} xUx^{-1} : U \text{ is a compact open subgroup of } G \right\}
\]
forms a basis of open neighborhoods at \( e \).

Let \( d \) be a metric on \( G \), and for each \( n \in \mathbb{N} \), let \( B\left(e, \frac{1}{n}\right) \) denote the open ball centered at \( e \) with radius \( \frac{1}{n} \). Define \( G_0 = G \in \mathcal{V} \), and for each \( n \in \mathbb{N} \), choose \( G_n \in \mathcal{V} \) such that
\[
G_n \subsetneq B\left(e, \frac{1}{n}\right) \cap \left( \bigcap_{i=0}^{n-1} G_i \right).
\]
The sequence $\langle G_n\rangle_{n \in \mathbb{N}_0}$ is then a strictly decreasing sequence of compact open normal subgroups of \( G \) satisfying the conditions of  Definition~(\ref{CVG}). This proves the claim.
\end{proof}

Let \( G \) be a CVG with a sequence of compact open normal subgroups \( \langle G_n\rangle_{n \in \mathbb{N}_0} \). The unitary dual of \( G \) is denoted by \( \widehat{G} \). Define \( G_{-1}^{\perp} = \emptyset \), and for each \( n \in \mathbb{N}_0 \), define $G_n^{\perp} := \{\chi \in \widehat{G} : \chi|_{G_n} = \mathrm{Id}\}$. The following result from \rm\cite[Section 2.1]{DelRod-2022Pre} provides a disjoint decomposition of the unitary dual \( \widehat{G} \) in terms of \( G_n^{\perp} \), \( n \in \mathbb{N}_0 \).

\begin{lmma}\label{Lemma 3.3}
Let $G$ be a compact Vilenkin group with a sequence of compact open normal subgroups $\langle G_n\rangle_{n \in \mathbb{N}_0}$. The sequence $\langle G_n^{\perp}\rangle_{n \in \mathbb{N}_0}$ consists of finite sets satisfying $G_{n-1}^{\perp} \subset G_n^{\perp}$ for all $n \in \mathbb{N}_0$, with $\widehat{G}$ as the disjoint union of $G_n^{\perp} \setminus G_{n-1}^{\perp}$ for $n \in \mathbb{N}_0$.
\end{lmma}

\begin{proof}
For each $n \in \mathbb{N}_0$, note that $G_n^{\perp}$ is in one-to-one correspondence with $\widehat{G/G_n}$ via the natural bijection $\Theta_n: G_n^{\perp} \to \widehat{G/G_n}$, defined by $\Theta_n(\chi)(x\bullet G_n) = \chi(x)$ for $x \in G$ and $\chi \in G_n^{\perp}$. Since $\left| G/G_n \right| = \left| G/G_0 \right| \cdot \left| G_0/G_1 \right| \cdots \left| G_{n-1}/G_n \right| < \infty$ for each $n \in \mathbb{N}_0$, it follows that $\widehat{G/G_n}$ is finite, and thus $G_n^{\perp}$ is finite. Moreover, since \( G_{n-1} \supset G_n \) for each \( n \in \mathbb{N} \), we have \( G_{n-1}^{\perp} \subset G_n^{\perp} \).

We next claim that each irreducible unitary representation $\chi: G \to GL_{d_\chi}(\mathbb{C})$ for $[\chi] \in \widehat{G}$ lies in $G_n^{\perp}$ for some $n \in \mathbb{N}_0$. It suffices to show that $G_n \subseteq \ker(\chi)$ for some $n$. Since $\chi$ is continuous and $\langle G_n\rangle_{n \in \mathbb{N}_0}$ forms a neighborhood basis at $e \in G$, there exists an $n \in \mathbb{N}_0$ such that $G_n \subseteq \chi^{-1}(V)$, where $V$ is a neighborhood of $I_{d_\chi}$ in $GL_{d_\chi}(\mathbb{C})$ such that $\{I_{d_\chi}\}$ is the only subgroup of $GL_{d_\chi}(\mathbb{C})$ contained in $V$. Hence, $\chi(G_n) = \left\{I_{d_\chi}\right\}$, establishing $G_n \subseteq \ker(\chi)$.
This, along with the fact that $G_{n-1}^{\perp} \subset G_n^{\perp}$ for all $n \in \mathbb{N}$, completes the proof.
\end{proof}
 
For a compact group $G$, let $R:G\rightarrow\mathcal{L}\left(L^2(G)\right)$ denote the right regular action of $G$ on $L^2(G)$, defined by $(R_x f)(y) = f(y\bullet x)$, for $x,y\in G$. For any subset $\mathcal{S} \subseteq \widehat{G}$, let $\mathcal{M}(\mathcal{S})$ denote the subspace of $L^2(G)$ finitely spanned by the matrix coefficients of representations in $\mathcal{S}$, with the convention $\mathcal{M}(\emptyset) := \{0\}$. Note that $\mathcal{O}(G) = \mathcal{M}\left(\widehat{G}\right)\subseteq C(G)$, and the Peter-Weyl theorem asserts that $\left\{\mathcal{M}(\chi) : \chi \in \widehat{G}\right\}$ forms a set of mutually orthogonal subspaces of $L^2(G)$, and $L^2(G)$ is the closed span of their direct sum, i.e.,
\[
L^2(G) = \widehat{\bigoplus}_{\chi \in \widehat{G}} \mathcal{M}(\chi).
\]

The following lemma follows directly from the Peter-Weyl theorem and the Frobenius reciprocity, which decompose $L^2(G/H)$ in terms of the matrix coefficients of $G$, where $G$ is a compact group and $H$ is a closed normal subgroup. Thus, we state it without proof.

\begin{lmma}\label{Strong Peter-Weyl}
Let $G$ be a compact group, and let $H$ be a closed normal subgroup of $G$. If $f \in L^2(G)$ satisfies $R_x f = f$ for all $x \in H$, then
$$f \in \widehat{\bigoplus_{\substack{\chi \in \widehat{G}, \chi|_H = \mathrm{Id}}}} \mathcal{M}(\chi).$$
\end{lmma}

\begin{lmma}\label{Filtration of MC}
Let $G$ be a compact Vilenkin group with a sequence of compact open normal subgroups $\langle G_n\rangle_{n \in \mathbb{N}_0}$. For each $n \in \mathbb{N}_0$, the space $\mathcal{M}(G_n^{\perp})$ is a finite-dimensional subspace of $L^2(G)$, and $\mathcal{M}(G_{n-1}^{\perp}) \subset \mathcal{M}(G_n^{\perp})$. Moreover, if $\mathcal{M}(G_n^{\perp}) \ominus \mathcal{M}(G_{n-1}^{\perp})$ denotes the $L^2$-orthogonal complement of $\mathcal{M}(G_{n-1}^{\perp})$ in $\mathcal{M}(G_n^{\perp})$, then
\[
\mathcal{M}(G_n^{\perp}) \ominus \mathcal{M}(G_{n-1}^{\perp}) = \mathcal{M}(G_n^{\perp} \setminus G_{n-1}^{\perp}).
\]
\end{lmma}

\begin{proof}
By Lemma~\ref{Lemma 3.3}, the sequence $\langle G_n^{\perp}\rangle_{n \in \mathbb{N}_0}$ consists of finite sets with $G_{n-1}^{\perp} \subset G_n^{\perp}$ for all $n \in \mathbb{N}_0$. Thus, $\mathcal{M}(G_n^{\perp})$ is a finite-dimensional subspace of $L^2(G)$ for each $n \in \mathbb{N}_0$, and $\mathcal{M}(G_{n-1}^{\perp}) \subset \mathcal{M}(G_n^{\perp})$. 

For each $n \in \mathbb{N}_0$, we have the decomposition
$\mathcal{M}(G_n^{\perp}) = \mathcal{M}(G_n^{\perp} \setminus G_{n-1}^{\perp}) \oplus \mathcal{M}(G_{n-1}^{\perp})$, where the subspaces $\mathcal{M}(G_n^{\perp} \setminus G_{n-1}^{\perp})$ and $\mathcal{M}(G_{n-1}^{\perp})$ are orthogonal in $L^2(G)$, hence it follows that
$\mathcal{M}(G_n^{\perp}) \ominus \mathcal{M}(G_{n-1}^{\perp}) = \mathcal{M}(G_n^{\perp} \setminus G_{n-1}^{\perp})$, as required.
\end{proof}

\begin{lmma}\label{Filtration}
Let $G$ be a compact Vilenkin group with a sequence of compact open normal subgroups $\{G_n\}_{n \in \mathbb{N}_0}$. For all $m, n \in \mathbb{N}_0$ with $m < n$, if $f \in \mathcal{M}(G_m^\perp)$ and $g \in \mathcal{M}(G_n^\perp \setminus G_{n-1}^\perp)$, then $f \cdot g \in \mathcal{M}(G_n^\perp \setminus G_{n-1}^\perp)$.
\end{lmma}

\begin{proof}
For each $x \in G_n$, we have
$$R_x(f \cdot g) = (R_x f) \cdot (R_x g) = f \cdot g.$$ Thus,  $f \cdot g \in \mathcal{M}(G_n^\perp)$ by Lemma~\ref{Strong Peter-Weyl}. We now claim that $f\cdot g \in \mathcal{M}(G_n^\perp) \ominus \mathcal{M}(G_{n-1}^\perp)$, for which it is sufficient to show that $f \cdot g$ is orthogonal in $L^2(G)$ to every $h \in \mathcal{M}(G_{n-1}^\perp)$.
Assume, for contradiction, that there exists $h \in \mathcal{M}(G_{n-1}^\perp)$ such that $\langle f \cdot g, h \rangle \neq 0$. Then, $\langle g, \overline{f} \cdot h \rangle \neq 0$. However, since $\overline{f} \cdot h \in L^2(G)$ satisfies, for each $x \in G_{n-1}$ that $R_x(\overline{f} \cdot h) = \overline{f} \cdot h$, it follows by Lemma~\ref{Strong Peter-Weyl} that $\overline{f} \cdot h \in \mathcal{M}(G_{n-1}^\perp)$. Hence, $\langle g, \overline{f} \cdot h \rangle \neq 0$ implies that $g \notin \mathcal{M}(G_n^\perp) \ominus \mathcal{M}(G_{n-1}^\perp) = \mathcal{M}(G_n^\perp \setminus G_{n-1}^\perp)$, a contradiction. Thus, the claim is established, and we conclude that $f \cdot g \in \mathcal{M}(G_n^\perp \setminus G_{n-1}^\perp)$.
\end{proof}
For a CVG $G$, we consider the associated canonical ergodic $C^*$-dynamical system $\left(C(G),G,\Delta\right)$. The Haar state on $G$ serves as an invariant state for the homogeneous space $C(G)$. The relevant covariant representation of this system is given by the triple $\left(L^2(G), \pi, R\right)$, where $L^2(G)$ is the GNS Hilbert space corresponding to the faithful Haar state on $G$, $\pi$ is the representation of $C(G)$ on $L^2(G)$ through left multiplication, and $R$ is the right regular representation.

\begin{thm}\label{SD of CVG}
The spectral dimension of every compact Vilenkin group is $0$.
\end{thm}

\begin{proof}
Let $G$ be a CVG, and let $\langle G_n\rangle_{n\in\mathbb{N}_0}$ be an associated sequence of compact open normal subgroups as in Definition \ref{CVG}. 
For each $s > 0$, consider the self-adjoint operator $D$ with compact resolvent specified by 
$$Df = \left( (n+1)^2 M_n \right)^{\frac{1}{s}} f,$$
for $f \in \mathcal{M}\left(G_n^{\perp} \setminus G_{n-1}^{\perp}\right),\, n \in \mathbb{N}_0$, where $M_n$ denotes the dimension of $\mathcal{M}\left(G_n^{\perp} \setminus G_{n-1}^{\perp}\right)$. We claim that for each $f \in \mathcal{O}(G)$, the commutator $[D, \pi(f)]$ extends to a bounded operator on $L^2(G)$. Since $\mathcal{O}(G) = \bigoplus_{n\in\mathbb{N}_0}\mathcal{M}\left(G_n^{\perp} \setminus G_{n-1}^{\perp}\right)$, to prove the above claim it suffices to show that for each $n\in\mathbb{N}_0$, and for each $f\in \mathcal{M}\left(G_n^{\perp} \setminus G_{n-1}^{\perp}\right)$, the commutator $[D,\pi(f)]$ extends to a bounded operator on $L^2(G)$.

Let $n \in \mathbb{N}_0$, and let $f\in \mathcal{M}\left(G_n^{\perp} \setminus G_{n-1}^{\perp}\right)$. For $m > n$ and $g \in \mathcal{M}\left(G_m^{\perp} \setminus G_{m-1}^{\perp}\right)$, we have by Lemma \ref{Filtration} that $f\cdot g \in \mathcal{M}\left(G_m^{\perp} \setminus G_{m-1}^{\perp}\right)$, and therefore
\[
[D, \pi(f)]g = D(f\cdot g) - \pi(f) D(g) = \left( (m+1)^2 M_m \right)^{\frac{1}{s}} f\cdot g - \pi(f) \left( (m+1)^2 M_m \right)^{\frac{1}{s}} g =0.
\]
Thus, $[D, \pi(f)]$ is in $\mathcal{L}(L^2(G))$, being zero on the complement of the finite-dimensional subspace $\bigoplus_{k=0}^n \mathcal{M}\left(G_k^{\perp} \setminus G_{k-1}^{\perp}\right)$ of $L^2(G)$.
Further, we have
\[
\mathrm{Tr} \, |D|^{-s} = \sum_{n \in \mathbb{N}_0} \frac{1}{(n+1)^2} < \infty,
\]
which implies that the spectral dimension \( \mathcal{S}\dim (C(G), G, \Delta) = 0 \).
\end{proof}

\begin{crlre}
The spectral dimension of every compact, totally disconnected, and metrizable topological group is $0$.
\end{crlre}

\begin{proof}
Let $G$ be a compact, totally disconnected, and metrizable topological group.

\textbf{Case 1:} If $G$ is finite, then the Hilbert space $L^2(G)$ is finite-dimensional. Consequently, any self-adjoint operator on $L^2(G)$ is bounded and $s$-summable for all $s > 0$. Thus, the spectral dimension of $G$ is $0$.

\textbf{Case 2:} If $G$ is infinite, then by Proposition~\ref{SD of CVG}, $G$ is a CVG. Hence, by Theorem~\ref{SD of CVG}, the spectral dimension of $G$ is $0$.
\end{proof}

We compute the dimension of a random walk on a CVG in the following.

\begin{thm}
The dimension of a (symmetric) random walk on a compact Vilenkin group is $0$.
\end{thm}

\begin{proof}
Let $G$ be a CVG. Denote the representation ring of $G$ by $\mathcal{R}(G)$, and for each $l \in \mathbb{N}_0$, let $\mathcal{R}_l(G)$ be the representation subring defined by $$\mathcal{R}_l(G) := \left\{\chi \in \mathcal{R}(G) : \chi|_{G_l} = \mathrm{Id}\right\}.$$ 
The bijection \(\Theta_l : G_l^\perp \to \widehat{G/G_l}\) given by
\[
\Theta_l(\chi)\left(x \bullet G_l\right) = \chi(x), \quad x \in G, \; \chi \in G_l^\perp,
\]
extends to a ring isomorphism \(\widetilde{\Theta_l} : \mathcal{R}_l(G) \to \mathcal{R}(G/G_l)\).

For a finite-dimensional representation \(\boldsymbol{\chi}\), let \(\mathfrak{m}(\boldsymbol{\chi})\) denote the multiplicity of the trivial representation in the decomposition of \(\boldsymbol{\chi}\) into irreducible components. Then, for any \(\chi \in \mathcal{R}_l(G)\) and \(k \in \mathbb{N}\), it is easy to see that
\[
\mathfrak{m}(\chi^{\otimes k}) = \mathfrak{m}\left(\widetilde{\Theta_l}(\chi)^{\otimes k}\right).
\]

Now, let \(\chi \in \mathcal{R}_{\mathrm{sym}}(G)\) with \(\dim(\chi) = k\). Then \(\chi \in \mathcal{R}_l(G)\) for some \(l \in \mathbb{N}_0\), and thus \(\widetilde{\Theta_l}(\chi) \in \mathcal{R}_{\mathrm{sym}}(G/G_l)\). Moreover,
\[
\mathfrak{h}\left(\chi_{11} + \cdots + \chi_{kk}\right)^{2n} = \mathfrak{m}(\chi^{\otimes n}) = \mathfrak{m}\left(\widetilde{\Theta_l}(\chi)^{\otimes n}\right) = \mathfrak{h}\left(\left(\widetilde{\Theta_l}(\chi)\right)_{11} + \cdots + \left(\widetilde{\Theta_l}(\chi)\right)_{kk}\right)^{2n},
\]
which implies \(p_n^\chi = p_n^{\widetilde{\Theta_l}(\chi)}\).

Since \(G/G_l\) is finite, it follows that $\displaystyle\overline{\lim}_{n \to \infty} \frac{-2 \log p_n^{\widetilde{\Theta_l}(\chi)}}{\log n} = 0$, and hence $\displaystyle\overline{\lim}_{n \to \infty} \frac{-2 \log p_n^\chi}{\log n} = 0$. Therefore, \(d_{RW}^G = 0\).
\end{proof}

The following corollary is immediate.

\begin{crlre}
The dimension of a (symmetric) random walk on a compact, totally disconnected, and metrizable topological group is $0$.
\end{crlre}

Now, we compute the GK dimension for a CVG.

\begin{thm}
Let $G$ be a compact Vilenkin group. Then, we have the folllowing:
$$\mathrm{GKdim}(\mathcal{O}(G)) = 0.$$
\end{thm}

\begin{proof}
Let $V$ be a finite-dimensional subspace of $\mathcal{O}(G)$ with basis $\mathcal{B} = \{f_i : i = 1, \ldots, m\}$. Since $\widehat{G} = \bigcup_{n \in \mathbb{N}_0} G_n^{\perp}$, we have $\mathcal{O}(G) = \operatorname{span}\left\{\bigcup_{n \in \mathbb{N}_0} \mathcal{M}(G_n^{\perp})\right\}$. By Lemma \ref{Filtration of MC}, $\mathcal{M}(G_{n-1}^{\perp}) \subset \mathcal{M}(G_n^{\perp})$ for each $n$, so $\mathcal{O}(G) = \bigcup_{n \in \mathbb{N}_0} \mathcal{M}(G_n^{\perp})$. We may thus choose $N$ such that $f_i \in \mathcal{M}(G_N^{\perp})$ for all $i$, giving $V \subseteq \mathcal{M}(G_N^{\perp})$.

We claim $\mathcal{M}(G_N^{\perp})$ is closed under products. For $g, h \in \mathcal{M}(G_N^{\perp})$, we have $R_x(g) = g$ and $R_x(h) = h$ for each $x \in G_N$, so $R_x(g \cdot h) = g \cdot h$. By Lemma \ref{Strong Peter-Weyl}, then $g \cdot h \in \mathcal{M}(G_N^{\perp})$, proving the claim.

Thus, for each $k \in \mathbb{N}_0$, we have $V^k \subseteq \mathcal{M}(G_N^{\perp})^k \subseteq \mathcal{M}(G_N^{\perp})$, implying
\[
\overline{\lim}_{k \to \infty} \frac{\log \dim(V^k)}{\log k} \leq \overline{\lim}_{k \to \infty} \frac{\log \dim(\mathcal{M}(G_N^{\perp}))}{\log k} = 0.
\]
Therefore, $\mathrm{GKdim}(\mathcal{O}(G)) = 0$.
\end{proof}

\begin{crlre}
Let $G$ be a compact, totally disconnected, and metrizable topological group. Then, one has $$\mathrm{GKdim}(\mathcal{O}(G)) = 0.$$
\end{crlre}

\brmrk
We proved that the spectral dimension, dimension of the symmetric random walk, and  the $\mathrm{GK}$ dimension of a compact, totally disconnected, metrizable topological group are all $0$. It would be   interesting  to  compute these dimensions for non-metrizable profinite groups. 
\ermrk
\section{$K$-groups}\label{$K$-groups}

A CVG $G$ with an associated sequence of compact open normal subgroups $\langle G_n\rangle_{n\in\mathbb{N}_0}$ can alternatively be identified as a topological group with the inverse limit of the system $\left\{\langle G/G_n\rangle_{n\in\mathbb{N}_0} , \langle \Phi_n\rangle_{n\in\mathbb{N}}\right\}$, where for each $n\in\mathbb{N}$, the transition map $$\Phi_n : G/G_n \rightarrow G/G_{n-1}$$ is given by $\Phi_n(x\bullet G_n) = x\bullet G_{n-1}$ for $x\in G$ (see \rm\cite[Thm 1.1.12]{RibZal-2010aa}). This identification is precisely given by the isomorphism 
$$\Phi : G \rightarrow \varprojlim_{n\in\mathbb{N}_0} (G/G_n);  \quad \Phi(x) = \langle x\bullet G_n\rangle_{n\in\mathbb{N}_0}, \mbox{ for } x\in G.$$  Under this identification $\Phi$, we henceforth write $G= \varprojlim_{n\in\mathbb{N}_0}(G/G_n)$. By the Gelfand-Naimark Theorem, $C(G)$  is then the inductive limit of the induced system $\left\{\langle C(G/G_n)\rangle_{n\in\mathbb{N}_0},\langle\Psi_n\rangle_{n\in\mathbb{N}}\right\}$, where 
$$\Psi_n : C(G/G_{n-1}) \rightarrow C(G/G_n); \quad (\Psi_n f)(x\bullet G_n) = f(x\bullet G_{n-1}) \mbox{  for all } f\in C(G/G_{n-1}) \mbox{  and } x\in G.$$ 
 The induced isomorphism $\Psi: \varinjlim_{n\in\mathbb{N}_0} C(G/G_n) \rightarrow C(G)$ is given by 
 $$\Psi (\langle f_n\rangle_{n\in\mathbb{N}_0})(x) = \lim_{n\to\infty} f_n (x\bullet G_n), \mbox{ for all } \langle f_n\rangle_{n\in\mathbb{N}_0} \in \varinjlim_{n\in\mathbb{N}_0} C(G/G_n) \mbox{ and } x\in G.$$  
Since each $C(G/G_n)$ is a finite-dimensional $C^*$-algebra, $C(G)$ is an approximately finite-dimensional (AF) algebra. Consequently, the identification $C(G) = \varinjlim_{n\in\mathbb{N}_0} C(G/G_n)$ implies that $K_i(C(G)) = \varinjlim_{n\in\mathbb{N}_0} K_i\left(C(G/G_n)\right)$ for $i=0,1$. Since each $G/G_n$ is a finite set, the spectrum of any unitary in $C(G/G_n)$  is not $S^1$, hence it can be connected to identity. This proves that   $K_1\left(C(G/G_n)\right) = 0$, and hence $K_1(C(G)) = 0$.

\begin{ppsn}\label{Generators of K_0}
Let $G$ be a compact Vilenkin group with an associated sequence of compact open normal subgroups $\langle G_n\rangle_{n \in \mathbb{N}_0}$. Then the group $K_0(C(G))$ is generated by the equivalence classes of the continuous functions $\mathbbm{1}_{x \bullet G_r}$ for all $r \in \mathbb{N}_0$ and $[x] \in G / G_r$.
\end{ppsn}

\prf
Consider the generating subset $\mathcal{K}$ of $\varinjlim_{n\in\mathbb{N}_0} K_0\left(C\left(G/G_n\right)\right)$ consisting of elements $\langle f_n\rangle_{n\in\mathbb{N}_0}$ such that there exists $r\in \mathbb{N}_0$ satisfying the recurrence relation $f_n=\Psi_n f_{n-1}$ for all $n>r$ with the initial condition $f_r=\mathbbm{1}_{\{x\bullet G_r\}}$ for some $[x]\in G/G_r$, or equivalently, we can express $f_n$ as
$$f_n = \sum_{y_1\in G_r/G_{r+1}}\sum_{y_2\in G_{r+1}/G_{r+2}}\cdots\sum_{y_n\in G_{n-1}/G_n} \mathbbm{1}_{\{x\bullet y_1\bullet y_2\bullet \cdots \bullet y_n\bullet G_n\}}.$$
Now, using the identification of $\varinjlim_{n\in\mathbb{N}_0} K_0\left(C\left(G/G_n\right)\right)$ with $K_0(C(G))$, we have $$\mathcal{K}=\left\{f\in K_0\left(C(G)\right):f=\mathbbm{1}_{x\bullet G_r} \text{ for } r\in\mathbb{N}_0 \text{ and } [x]\in G/G_r\right\}.$$
\qed\\ 
In what follows, we consider the commutative CVG $\mathbb{Z}_p$ and the noncommutative CVG $\mathbb{H}_d(\mathbb{Z}_p)$, and provide explicit decompositions of the generators of their $K_0$-groups in terms of matrix coefficients.

\subsection{Generators of $K_0\left(C\left(\mathbb{Z}_p\right)\right)$ as a span of characters}
Since $\mathbb{Z}_p$ is commutative, each element of $\widehat{\mathbb{Z}_p}$ corresponds to a one-dimensional unitary representation of $\mathbb{Z}_p$, identified with a character, i.e., a continuous group homomorphism from $\mathbb{Z}_p$ to the circle group $\mathbb{T}$. Conversely, each character of $\mathbb{Z}_p$ corresponds to an element in $\widehat{\mathbb{Z}_p}$. Under this identification, the unitary dual $\widehat{\mathbb{Z}_p}$ forms a group under pointwise multiplication, isomorphic to the Prüfer group $\mathbb{Z}(p^{\infty}) = \{ e^{\frac{2\pi im}{p^n}} : 0 \leq m < p^n, n \in \mathbb{N} \}$. For any character $\chi \in \widehat{\mathbb{Z}_p}$, we have 
$$\chi(1) = e^{\frac{2\pi im}{p^n}} \mbox{ for some } 0 \leq m < p^n \mbox{ and } n \in \mathbb{N}.$$ 
Moreover, we get 
$$\{\chi(1) : \chi \in \widehat{\mathbb{Z}_p}\} = \mathbb{Z}(p^{\infty}).$$ Thus, $\widehat{\mathbb{Z}_p}$ is isomorphic to $\mathbb{Z}(p^{\infty})$ via the map $\chi \mapsto \chi(1)$. Let 
$$S = \{(1,0)\} \cup \{(m,n) \in \mathbb{N} : m < p^n,\, p \nmid m\},$$ which bijectively parametrizes the elements of $\mathbb{Z}(p^{\infty})$. Henceforth, we adopt the notation $\chi_{m,n}$ for all $(m,n)\in S$ to represent the character of $\mathbb{Z}_p$, where $\chi_{m,n}(1)=e^{\frac{2\pi im}{p^n}}$.

We recall from  \rm\cite{Mon-1952aa} a  continuous mapping from $\mathbb{Z}_p$ to the interval $[0,1]$, known as the Monna map. This mapping, denoted as $T:\mathbb{Z}_p\rightarrow [0,1]$, is defined by
\begin{IEEEeqnarray}{rCl}\label{Eq 4.1}
T\left(\sum_{k=0}^\infty x_k p^k\right)=\sum_{k=0}^\infty\frac{x_k}{p^{k+1}},
\end{IEEEeqnarray} 
where $0\leq a_k\leq p-1$. The Monna map $T$ is continuous with respect to the $p$-adic metric on $\mathbb{Z}_p$ and the standard Euclidean metric on $[0,1]$. Furthermore, it preserves measures, considering the Haar measure $\mu$ on $\mathbb{Z}_p$ and the Lebesgue measure $\nu$ on $[0,1]$. Further, note that the set 
$$E=\{a\in [0,1]: a \text{ has multiple base } p \text{ representations}\}$$ is countable and thus has measure zero. Consequently, when we restrict the map $T$ to $\mathbb{Z}_p\setminus T^{-1}[E]$, it becomes a bijection. As a result, the induced map $\widetilde{T}:L^2([0,1])\rightarrow L^2(\mathbb{Z}_p)$ defined as $\widetilde{T}(f):=f\circ T$, where $f\in L^2([0,1])$, is unitary. The following proposition provides the decomposition of the generating elements of $K_0(\mathbb{Z}_p)$ in terms of characters.

\begin{ppsn}\label{Decomposition for p-adic integers}
For any $r\in\mathbb{N}$ and $x\in\mathbb{Z}_p/{p^r\mathbb{Z}_p}$, we have $$\mathbbm{1}_{x+p^r\mathbb{Z}_p}=\sum_{(m,n)\in S,\,n\leq r}\frac{\overline{\left(\chi_{m,n}(1)\right)}^x}{p^r}\chi_{m,n}.$$
\end{ppsn}

\begin{proof}
Let $r\in\mathbb{N}$ and consider the unique representative of $x$ in $\{0,1,\ldots, p^r-1\}$. Since for each $y\in p^r\mathbb{Z}_p$, we have $R_y\left(\mathbbm{1}_{x+p^r\mathbb{Z}_p}\right) = \mathbbm{1}_{x+p^r\mathbb{Z}_p}$, by Lemma \ref{Strong Peter-Weyl}, $\mathbbm{1}_{x + p^r \mathbb{Z}_p}$ lies in the linear span of $\chi_{m, n}$, where $(m, n) \in S$ and $n \leq r$. note that $\widetilde{T}^{-1}\left(\mathbbm{1}_{x+p^r\mathbb{Z}_p}\right)=\mathbbm{1}_{\left[T(x),T(x)+\frac{1}{p^r}\right]}$.

For $(m, n) \in S$ with $n \leq r$, we find
$$\widetilde{T}^{-1}\left(\chi_{m,n}\right)=\lim_{N\rightarrow\infty}\sum_{l_0,l_1,\ldots,l_N=0}^{p-1} \left(\chi_{m,n}(1)\right)^{l_0+l_1p+\cdots+l_Np^N}\mathbbm{1}_{\left[\frac{l_0}{p} + \frac{l_1}{p^2} + \cdots + \frac{l_N}{p^{N+1}}, \frac{l_0}{p} + \frac{l_1}{p^2} +\cdots + \frac{l_N +1}{p^{N+1}}\right)}.$$

The $L^2$-inner product $\langle\mathbbm{1}_{x+p^r\mathbb{Z}_p},\chi_{m,n}\rangle$ simplifies to
\begin{align*}
&\left\langle \mathbbm{1}_{x+p^r\mathbb{Z}_p},\chi_{m,n}\right\rangle\\
&=\left\langle \widetilde{T}^{-1}\left(\mathbbm{1}_{x+p^r\mathbb{Z}_p}\right),\widetilde{T}^{-1}\left(\chi_{m,n}\right)\right\rangle\\
&=\left\langle\mathbbm{1}_{\left[T(x),T(x)+\frac{1}{p^r}\right]}, \lim_{N\rightarrow\infty}\sum_{l_0,l_1,\ldots,l_N=0}^{p-1} \left(\chi_{m,n}(1)\right)^{l_0+l_1p+\cdots+l_Np^N}\mathbbm{1}_{\left[\frac{l_0}{p} + \cdots + \frac{l_N}{p^{N+1}}, \frac{l_0}{p} +\cdots + \frac{l_N +1}{p^{N+1}}\right)}\right\rangle\\
&=\lim_{N\rightarrow\infty}\int_{T(x)}^{T(x)+\frac{1}{p^r}}\sum_{l_0,l_1,\ldots,l_N=0}^{p-1} \overline{\left(\chi_{m,n}(1)\right)}^{l_0+l_1p+\cdots+l_Np^N}\mathbbm{1}_{\left[\frac{l_0}{p} + \cdots + \frac{l_N}{p^{N+1}}, \frac{l_0}{p} +\cdots + \frac{l_N +1}{p^{N+1}}\right)}\, d\nu\\
&=\lim_{N\rightarrow\infty}\sum_{l_r,\ldots,l_N=0}^{p-1}\overline{\left(\chi_{m,n}(1)\right)}^{x + \sum_{k=r}^N l_kp^k}\int_{T(x)}^{T(x)+\frac{1}{p^r}} \mathbbm{1}_{\left[T(x) + \sum_{k=r}^N\frac{l_k}{p^{k+1}}, T(x) + \sum_{k=r}^N\frac{l_k}{p^{k+1}} + \frac{1}{p^{N+1}}\right)}\,d\nu\\
&=\overline{\left(\chi_{m,n}(1)\right)}^x \lim_{N\rightarrow\infty}\frac{1}{p^{N+1}}\sum_{l_r,\ldots, l_N=0}^{p-1}\overline{\left(\chi_{m,n}(1)\right)}^{\sum_{k=r}^N l_kp^k}\\
&=\overline{\left(\chi_{m,n}(1)\right)}^x \lim_{N\rightarrow\infty}\frac{1}{p^{N+1}}\prod_{k=r}^N\left(\sum_{l=0}^{p-1}\overline{\left(\chi_{m,n}(1)\right)}^{lp^k}\right) = \overline{\left(\chi_{m,n}(1)\right)}^x\lim_{N\rightarrow\infty}\frac{p^{N-r+1}}{p^{N+1}}=\frac{\overline{\left(\chi_{m,n}(1)\right)}^x}{p^r},
\end{align*}
and this completes the proof.
\end{proof}

\subsection{Generators of $K_0\left(\mathbb{H}_d\left(\mathbb{Z}_p\right)\right)$ in terms of Matrix coefficients}

\begin{dfn}
Let $R$ be a ring with unity $1$, and $S$ a subring of $R$. For $d \in \mathbb{N}$, the $(2d+1)$-dimensional Heisenberg group over $S$, denoted by $\mathbb{H}_d(S)$, is defined as
\[
\mathbb{H}_d(S) = \left\{[x,y,z]:=
\begin{bmatrix}
  1 & x^t & z \\
  0 & I_d & y \\
  0 & 0 & 1
\end{bmatrix}
: x, y \in S^d, z \in S
\right\},
\]
with group operation defined by matrix multiplication. Here, $I_d$ is the $d \times d$ identity matrix, and $x^t$ denotes the transpose of $x$.
\end{dfn}

For $R = \mathbb{Z}_p$, the $p$-adic integers, $\mathbb{H}_d(\mathbb{Z}_p)$ forms a CVG with a sequence of compact open normal subgroups defined by $G_n := \mathbb{H}_d(p^n \mathbb{Z}_p)$ for $n \in \mathbb{N}_0$. The map
\[
\mathcal{I}_n: \mathbb{H}_d(p^n \mathbb{Z}_p) \to \left( \mathbb{Z}/p \mathbb{Z} \right)^{2d+1},
\]
given by
\[
\mathcal{I}_n\left([x,y,z]\right) = \left( \frac{x(1)}{p^n} + p\mathbb{Z}, \ldots, \frac{x(d)}{p^n} + p\mathbb{Z}, \frac{y(1)}{p^n} + p\mathbb{Z}, \ldots, \frac{y(d)}{p^n} + p\mathbb{Z}, \frac{z}{p^n} + p\mathbb{Z} \right),
\]
for $x = (x(1), \ldots, x(d))^t$, $y = (y(1), \ldots, y(d))^t \in (p^n \mathbb{Z}_p)^d$, and $z \in p^n \mathbb{Z}_p$, is a surjective homomorphism with kernel $\mathbb{H}_d(p^{n+1} \mathbb{Z}_p)$. Hence, it induces an isomorphism:
\[
\mathbb{H}_d(p^n \mathbb{Z}_p) / \mathbb{H}_d(p^{n+1} \mathbb{Z}_p) \cong \left( \mathbb{Z}/p \mathbb{Z} \right)^{2d+1}.
\]
Similarly, we have 
\[
\mathbb{H}_d(\mathbb{Z}_p) / \mathbb{H}_d(p^n \mathbb{Z}_p) \cong \mathbb{H}_d\left( \mathbb{Z}_p / p^n \mathbb{Z}_p \right).
\]

Using the topological isomorphism $\iota: \mathbb{H}_d(\mathbb{Z}_p) \to \mathbb{Z}_p^{2d+1}$ defined by 
\[
\iota([x, y, z]) = (x, y, z), \quad x, y \in \mathbb{Z}_p^d, \; z \in \mathbb{Z}_p,
\]
we identify $\mathbb{H}_d(\mathbb{Z}_p)$ with $\mathbb{Z}_p^{2d+1}$ as topological spaces. The center of $\mathbb{H}_d(\mathbb{Z}_p)$ is given by 
\[
Z(\mathbb{H}_d(\mathbb{Z}_p)) = \{[0, 0, z] : z \in \mathbb{Z}_p\} \cong \mathbb{Z}_p.
\]
The quotient group $\mathbb{H}_d(\mathbb{Z}_p) / Z(\mathbb{H}_d(\mathbb{Z}_p))$ is isomorphic to $\mathbb{Z}_p^{2d}$ via the map 
\[
[x, y, z] \cdot Z(\mathbb{H}_d(\mathbb{Z}_p)) \mapsto (x^t, y^t), \quad [x, y, z] \in \mathbb{H}_d(\mathbb{Z}_p).
\]

Thus, the Haar measure on $Z(\mathbb{H}_d(\mathbb{Z}_p))$ is $\mu$, and the Haar measure on $\mathbb{H}_d(\mathbb{Z}_p) / Z(\mathbb{H}_d(\mathbb{Z}_p))$ is the product measure $\mu^{\otimes 2d}$. By applying the Quotient Integral Formula, the Haar measure on $\mathbb{H}_d(\mathbb{Z}_p)$, via the isomorphism $\iota$, is then the product measure $\mu^{\otimes (2d+1)}$. Hence, similar to the case of $\mathbb{Z}_p$, the map $\mathcal{T}: L^2\left([0,1]^{2d+1}\right) \rightarrow L^2\left(\mathbb{H}_d\left(\mathbb{Z}_p\right)\right)$, defined by
$$\mathcal{T}(f)([x,y,z]) = f\left(T(x(1)), \ldots, T(x(d)), T(y(1)), \ldots, T(y(d)), Tz\right)$$
for $f \in L^2\left([0,1]^{2d+1}\right)$, $x = (x(1), \ldots, x(d))^t$, $y = (y(1), \ldots, y(d))^t \in \mathbb{Z}_p^d$, and $z \in \mathbb{Z}_p$, is unitary.

For \( x := \sum_{k=i}^{\infty} x_k p^k + \mathbb{Z}_p \in \mathbb{Q}_p / \mathbb{Z}_p \), we define
\[
\{x\}_p := \begin{cases} 
\sum_{k=i}^{-1} x_k p^k & \text{if } i \leq -1, \\ 
0 & \text{if } i \geq 0. 
\end{cases}
\]
We denote by \( |\cdot|_p \) the standard \(\ell_{\infty}\)-norm on \( \mathbb{Q}_p^n \) induced by the \( p \)-adic norm on \( \mathbb{Q}_p \). Identifying each equivalence class in $\widehat{\mathbb{Z}_p} \cong \mathbb{Q}_p / \mathbb{Z}_p$ with its associated representative from the complete system of representatives
\[
\{1\} \cup \left\{ \sum_{k=1}^{\infty} \frac{\lambda_k}{p^k} : 0 \leq \lambda_k \leq p-1,\, 1 \leq \sum_{k=1}^{\infty} \lambda_k < \infty \right\},
\]
and considering the following norm on $\widehat{\mathbb{Z}_p}^{2d+1}$ defined by $\|(\alpha,\beta,\gamma)\|_p := \max \{|\alpha|_p,|\beta|_p, |\gamma|_p\}$ for $(\alpha,\beta,\gamma)\in \widehat{\mathbb{Z}_p}^{2d+1}$, we have the following explicit description of the unitary dual of $\mathbb{H}_d(\mathbb{Z}_p)$.

\begin{thm}{\rm\cite{Vel-2024Pre}\label{Representations of pHG}}
The unitary dual of the $(2d+1)$-dimensional Heisenberg group $\mathbb{H}_d(\mathbb{Z}_p)$ can be identified with the set
\[
\left\{ (\alpha,\beta,\gamma) \in \widehat{\mathbb{Z}_p}^{2d+1} : (\alpha, \beta) \in \mathbb{Q}_p^{2d} / |\gamma|_p^{-1} \mathbb{Z}_p^{2d} \right\}
\]
as follows: Each element $(\alpha,\beta,\gamma) \in \widehat{\mathbb{Z}_p}^{2d+1}$, where $(\alpha, \beta) \in \mathbb{Q}_p^{2d} / |\gamma|_p^{-1} \mathbb{Z}_p^{2d}$, corresponds to the representation class $\left[\chi_{(\alpha,\beta,\gamma)}\right] \in \widehat{\mathbb{H}_d(\mathbb{Z}_p)}$, where the representation $\chi_{(\alpha,\beta,\gamma)}$ of $\mathbb{H}_d(\mathbb{Z}_p)$ acts on the finite-dimensional subspace $\mathcal{H}_{\gamma}$ of $L^2(\mathbb{Z}_p^d)$, defined as
\[
\mathcal{H}_{\gamma} = \mathrm{span}_{\mathbb{C}} \left\{ |\gamma|_p^{d/2} \cdot \mathbbm{1}_{k + |\gamma|_p \mathbb{Z}_p^d} : k \in \mathbb{Z}_p^d / |\gamma|_p \mathbb{Z}_p^d \right\},
\]
where $\dim(\mathcal{H}_{\gamma}) = |\gamma|_p^d$, and the action is given by
\[
(\chi_{(\alpha,\beta,\gamma)}[x,y,z]\varphi)(w) = e^{2\pi i \{\alpha x + \beta y + \gamma (z + wy)\}_p} \varphi(w + x), \quad \text{ for }\varphi \in \mathcal{H}_{\gamma} \text{ and } w\in \mathbb{Z}_p^d.
\]
With this identification, and using the basis $\left\{ |\gamma|_p^{d/2} \cdot \mathbbm{1}_{k +|\gamma|_p\mathbb{Z}_p^d} : k \in \mathbb{Z}_p^d / |\gamma|_p \mathbb{Z}_p^d \right\}$ of $\mathcal{H}_{\gamma}$, the matrix coefficients are given by
\[
\left(\chi_{(\alpha,\beta,\gamma)}\right)_{k,k'}([x,y,z]) = e^{2\pi i \left\{ \gamma(z + k'y) + (x\alpha + y\beta) \right\}_p} \mathbbm{1}_{k-k' + |\gamma|_p\mathbb{Z}_p^d}(x), \quad [x,y,z] \in \mathbb{H}_d(\mathbb{Z}_p).
\]
Moreover, 
$$\mathbb{H}_d(p^n\mathbb{Z}_p)^{\perp} = \left\{(\alpha,\beta,\gamma)\in\widehat{\mathbb{Z}_p}^{2d+1} : (\alpha, \beta)\in \mathbb{Q}^{2d}/{|\gamma|_p^{-1}\mathbb{Z}_p^{2d}},\,\|(\alpha,\beta,\gamma)\|_p \leq p^n\right\}, \text{and}$$
$$\mathbb{H}_d(p^n\mathbb{Z}_p)^{\perp} \setminus \mathbb{H}_d(p^{n-1}\mathbb{Z}_p)^{\perp} = \left\{(\alpha,\beta,\gamma)\in \widehat{\mathbb{Z}_p}^{2d+1} : (\alpha,\beta) \in \mathbb{Q}^{2d}/{|\gamma|_p^{-1}\mathbb{Z}_p^{2d}},\, \|(\alpha,\beta,\gamma)\|_p = p^n\right\}.$$
\end{thm}

For \( \gamma \in \widehat{\mathbb{Z}_p} \), let \( P_\gamma : \mathbb{Z}_p^d \to \mathbb{Z}_p^d / |\gamma|_p \mathbb{Z}_p^d \) denote the natural quotient map. For simplicity, we denote the vector \( (\alpha, \beta, \gamma) \) by \( \mathbf{\zeta} \). The following proposition expresses the generators of \( K_0\big( \mathbb{H}_d(\mathbb{Z}_p) \big) \) in terms of matrix coefficients.

\begin{ppsn}
For any \( r \in \mathbb{N} \) and \( \boldsymbol{v} := [x, y, z] \in \mathbb{H}_d\left( \mathbb{Z}_p \right) \), we have
\[
\mathbbm{1}_{\boldsymbol{v} \cdot \mathbb{H}_d\left( p^r \mathbb{Z}_p \right)} = \sum_{\substack{\gamma \in \mathbb{Q}_p / \mathbb{Z}_p \\ |\gamma|_p \leq p^r}} \left(\sum_{\substack{(\alpha, \beta) \in \mathbb{Q}_p^{2d} / |\gamma|_p^{-1} \mathbb{Z}_p^{2d} \\ |\alpha|_p, |\beta|_p \leq p^r}} \left(\sum_{k \in \mathbb{Z}_p^d / |\gamma|_p \mathbb{Z}_p^d} \frac{|\gamma|_p^d \cdot e^{-2\pi i \{ \gamma(z + ky) + x \alpha + y \beta \}_p}}{p^{r(2d+1)}} \left( \chi_{\mathbf{\zeta}} \right)_{k + P_{\gamma}(x), k}\right)\right).
\]
\end{ppsn}

\begin{proof}
The function \( \mathbbm{1}_{[x, y, z] \cdot \mathbb{H}_d\left( p^r \mathbb{Z}_p \right)} \in L^2\left( \mathbb{H}_d\left( \mathbb{Z}_p \right) \right) \) is invariant under the right regular action of the subgroup \( \mathbb{H}_d\left( p^r \mathbb{Z}_p \right) \). By Lemma \ref{Strong Peter-Weyl} and Theorem \ref{Representations of pHG}, it then follows that
$$\mathbbm{1}_{[x,y,z]\cdot \mathbb{H}_d\left(p^r\mathbb{Z}_p\right)} = \sum_{(\alpha,\beta,\gamma)\in\mathbb{H}_d\left(p^r\mathbb{Z}_p\right)^{\perp}} \left(\sum_{k,k'\in \mathbb{Z}_p^d/{|\gamma|_p \mathbb{Z}_p^d}} \frac{\langle \mathbbm{1}_{[x,y,z]\cdot \mathbb{H}_d\left(p^r\mathbb{Z}_p\right)} , \left(\chi_{(\alpha,\beta,\gamma)}\right)_{kk'}\rangle}{\|\left(\chi_{(\alpha,\beta,\gamma)}\right)_{kk'}\|^2} \left(\chi_{(\alpha,\beta,\gamma)}\right)_{kk'}\right)$$

Let $(\alpha,\beta,\gamma)\in\mathbb{H}_d\left(p^r\mathbb{Z}_p\right)^{\perp}$, i.e., $(\alpha,\beta)\in\mathbb{Q}_p^{2d}/{|\gamma|_p^{-1}\mathbb{Z}_p^{2d}}$, $\gamma\in\mathbb{Q}_p/{\mathbb{Z}_p}$, and $\|\left(\alpha,\beta,\gamma\right)\|_p\leq p^r$. Let $k,k'\in \mathbb{Z}_p^d/{|\gamma|_p\mathbb{Z}_p^d}$.

The \( L^2 \)-norm of the matrix coefficient \( \left( \chi_{(\alpha, \beta, \gamma)} \right)_{k k'} \) is
\begin{align*}
\|\left(\chi_{(\alpha,\beta,\gamma)}\right)_{kk'}\|^2 &= \int \mathbbm{1}_{k-k' + |\gamma|_p\mathbb{Z}_p^d}(a)\, d\mu^{\otimes (2d+1)}(a,b,c)\\
&= \int \mathbbm{1}_{k-k'+|\gamma|_p\mathbb{Z}_p^d}(a)\, d\mu^{\otimes d}(a) = \mu^{\otimes d}\left(k-k' + |\gamma|_p\mathbb{Z}_p^d\right) = \frac{1}{|\gamma|_p^d}.
\end{align*}

Now, we compute the inner products as follows:
\begin{align*}
I &:=\langle \mathbbm{1}_{[x,y,z]\cdot \mathbb{H}_d\left(p^r\mathbb{Z}_p\right)} , \left(\chi_{(\alpha,\beta,\gamma)}\right)_{kk'}\rangle\\
&=\int_{[x,y,z]\cdot \mathbb{H}_d\left(p^r\mathbb{Z}_p\right)} e^{-2\pi i\{\gamma(c+k' b) + a\alpha + b\beta\}_p}\cdot \mathbbm{1}_{k-k' + |\gamma|_p\mathbb{Z}_p^d}(a)\,d\mu^{\otimes (2d+1)}(a,b,c).
\end{align*}
Using the transformation $[a,b,c] = [x,y,z]\cdot [a_1,b_1,c_1]$, we get
\begin{align*}
I= \int_{\mathbb{H}_d\left(p^r\mathbb{Z}_p\right)} e^{-2\pi i \{\gamma\left(c_1 + x^t b_1 +z + k'(b_1 + y)\right) + (a_1 +x)\alpha + (b_1 +y)\beta\}_p}\cdot \mathbbm{1}_{k-k' + |\gamma|_p\mathbb{Z}_p^d}(a_1 + x)\, d\mu^{\otimes (2d+1)}(a_1,b_1,c_1).
\end{align*}
Replacing the variable notation $(a_1,b_1,c_1)$ by $(a,b,c)$ and applying the Quotient Integral Formula, we get
\begin{align*}
I &= \int_{(a,b)\in p^r\mathbb{Z}_p^{2d}}\int_{c\in p^r\mathbb{Z}_p} e^{-2\pi i\{\gamma\left(c + x^t b + z + k'(b+y)\right) + (a+x)\alpha + (b+y)\beta\}_p}\cdot \mathbbm{1}_{-x+k-k' + |\gamma|_p\mathbb{Z}_p^d}(a)\,d\mu(c)\,d\mu^{\otimes (2d)}(a,b).
\end{align*}
If $\chi_{\gamma}$ denotes the character on $\mathbb{Z}_p$ with $\chi_{\gamma}(1) = e^{2\pi i \gamma}$, it is then shown in the proof of Proposition \ref{Decomposition for p-adic integers} that 
\begin{equation}\label{eq1}
\int_{p^r\mathbb{Z}_p} e^{-2\pi i\{\gamma c\}_p}\, d\mu (c) = \langle \mathbbm{1}_{p^r\mathbb{Z}_p} , \chi_{\gamma}\rangle = \frac{1}{p^r},
\end{equation}
and this gives
\begin{IEEEeqnarray*}{lCl}
I = \frac{1}{p^r}\int_{p^r\mathbb{Z}_p^{2d}} e^{-2\pi i \{\gamma\left(x^t b + z + k'(b+y)\right) + (a+x)\alpha + (b+y)\beta\}_p}\cdot \mathbbm{1}_{-x + k-k' + |\gamma|_p\mathbb{Z}_p^d}(a)\, d\mu^{\otimes (2d)}(a,b)\\
= \frac{e^{-2\pi i\{\gamma(z+k'y) + x\alpha + y\beta\}_p}}{p^r} \int_{p^r\mathbb{Z}_p^{2d}} e^{-2\pi i\{\gamma (x^t + k')b + a\alpha + b\beta\}_p} \cdot \mathbbm{1}_{-x+k-k'+|\gamma|_p \mathbb{Z}_p^d}(a)\, d\mu^{\otimes (2d)}(a,b)\\
= \frac{e^{-2\pi i\{\gamma(z+k'y) + x\alpha + y\beta\}_p}}{p^r} \int_{p^r\mathbb{Z}_p^d} e^{-2\pi i\{a\alpha\}_p} \cdot \mathbbm{1}_{-x+k-k'+|\gamma|_p \mathbb{Z}_p^d}(a) \int_{p^r\mathbb{Z}_p^d} e^{-2\pi i\{\left(\gamma (x^t + k') + \beta\right)b\}_p}\, d\mu^{\otimes d}(b)\, d\mu^{\otimes d}(a).
\end{IEEEeqnarray*}
Using equation (\ref{eq1}) on each co-ordinate of $p^r\mathbb{Z}_p^d$, it then follows that $$\int_{p^r\mathbb{Z}_p^d} e^{-2\pi i\{\left(\gamma (x^t + k') + \beta\right)b\}_p}\, d\mu^{\otimes d}(b) = \frac{1}{p^{rd}}.$$ 
Hence we have
\begin{align*}
I &= \frac{e^{-2\pi i\{\gamma(z+k'y) + x\alpha + y\beta\}_p}}{p^{r(d+1)}} \int_{p^r\mathbb{Z}_p^d} e^{-2\pi i\{a\alpha\}_p} \cdot \mathbbm{1}_{-x+k-k'+|\gamma|_p \mathbb{Z}_p^d}(a)\, d\mu^{\otimes d}(a)\\
&= \frac{e^{-2\pi i\{\gamma(z+k'y) + x\alpha + y\beta\}_p}}{p^{r(d+1)}} \int_{p^r\mathbb{Z}_p^d\cap\left(-x+k-k'+|\gamma|_p\mathbb{Z}_p^d\right)}e^{-2\pi i\{a\alpha\}_p}\, d\mu^{\otimes d}(a).
\end{align*}
If $k\neq P_{\gamma}(x) + k'$, then $p^r\mathbb{Z}_p^d\cap\left(-x+k-k'+|\gamma|_p\mathbb{Z}_p^d\right)$ is empty, and hence $I=0$. On the other hand, if $k= P_{\gamma}(x) + k'$, then $p^r\mathbb{Z}_p^d\cap\left(-x+k-k'+|\gamma|_p\mathbb{Z}_p^d\right) = p^r\mathbb{Z}_p^d$, giving us 
$$I=\frac{e^{-2\pi i\{\gamma(z+k'y) + x\alpha + y\beta\}_p}}{p^{r(d+1)}} \int_{p^r\mathbb{Z}_p^d}e^{-2\pi i\{a\alpha\}_p}\, d\mu^{\otimes d}(a) = \frac{e^{-2\pi i\{\gamma(z+k'y) + x\alpha + y\beta\}_p}}{p^{r(2d+1)}}.$$ 
This leads to the decomposition as stated.
\end{proof}

\section{On Spectral triples}\label{Spectral triple}
 A well known example of nontrivial even spectral triple is the following. Consider the Toeplitz algebra $\scrt$, which is the $^*$-closed subalgebra of $\mathcal{L}\left(\ell^2\left(\mathbb{N}_0\right)\right)$ generated by by the unilateral shift $S$ defined by $Se_n = e_{n+1}$, $n\in\mathbb{N}_0$. We take $\mathcal{H} = \ell^2\left(\mathbb{N}_0\right)\oplus \ell^2\left(\mathbb{N}_0\right)$, which is naturally $\mathbb{Z}/2$-graded, with grading operator $\Gamma = \begin{pmatrix}
1 & 0 \\
0 & -1
\end{pmatrix}$. We consider the Dirac operator $\begin{pmatrix}
0 & D \\
D^* & 0
\end{pmatrix}$, where $D$ is the weighted shift operator defined by $De_n = (n+1)e_{n+1}$, $n\in\mathbb{N}_0$. Then, the corresponding spectral triple is a nontrivial spectral triple. One can construct similar nontrivial spectral triples, where the corresponding Dirac operators are weighted shift on the respective Hilbert space on the noncommutative torus $\mathcal{A}_{\theta}$, on the Quantum Disk $\mathcal{D}_q$, and on quantum $3$-sphere $S_{\theta}^3$. In this section, we show the non-existence of a general class of such even Spectral triples on $\mathbb{Z}_p$.
We begin with the following characterization of an equivariant even spectral triple for the $^*$-subalgebra $\mathcal{O}(G)$ of $C(G)$, where $G$ is a compact abelian group.

\begin{ppsn} \label{Equivariance}
Let $G$ be a compact abelian group, and let $U$ be the left regular action of $G$ on $L^2(G)$. Let
$$
\left(\mathcal{O}(G), \pi\oplus\pi,  
\begin{pmatrix}
0 & D \\
D^* & 0
\end{pmatrix}, \begin{pmatrix}
1 & 0 \\
0 & -1
\end{pmatrix}
\right)
$$ 
be an equivariant even spectral triple of $C(G)$, where $\pi$ is the GNS representation of $\mathcal{O}(G)$ on $L^2(G)$. Let $\{\chi_n : n\in\mathbb{N}_0\}$ be the set of characters of $G$. Then, for each $n \in \mathbb{N}_0$, there exists $\lambda_n \in \mathbb{C}$ such that $D\chi_n = \lambda_n \chi_n$.
\end{ppsn}

\begin{proof}
Equivariance of the spectral triple implies that for all $x\in G$, $U_x^* D U_x = D$. Fix $n\in\mathbb{N}_0$. Then
\begin{align*}
U_x^*DU_x\chi_n = U_x^* D\left(\overline{\chi_n(x)}\chi_n\right) &= \overline{\chi_n(x)} U_x^*\left(\sum_{m\in\mathbb{N}_0}\left\langle D\chi_n,\chi_m\right\rangle\chi_m\right)\\
&= \overline{\chi_n(x)}\sum_{m\in\mathbb{N}_0}\left\langle D\chi_n,\chi_m\right\rangle\chi_m(x)\chi_m.
\end{align*}
For each $m\in\mathbb{N}_0$, equating coefficients of $\chi_m$ in $U_x^* D U_x \chi_n$ and $D\chi_n$ gives
\begin{IEEEeqnarray}{rCl} \label{equn:equivariance}
\overline{\chi_n(x)}\chi_m(x) \langle D\chi_n, \chi_m \rangle = \langle D\chi_n, \chi_m \rangle.
\end{IEEEeqnarray}
For $m\neq n$, choosing $x\in G$ such that $\chi_m(x) \neq \chi_n(x)$ forces $\langle D\chi_n, \chi_m \rangle = 0$. Hence,
$$
D\chi_n = \langle D\chi_n, \chi_n \rangle \chi_n,
$$
where setting $\lambda_n = \langle D\chi_n, \chi_n \rangle$ completes the proof.
\end{proof}

The following lemma provides a natural basis of $qL^2\left(\mathbb{Z}_p\right)$, where $[q]$ is a generator of $K_0\left(C\left(\mathbb{Z}_p\right)\right)$ as described in Proposition \ref{Generators of K_0}.

\begin{lmma}\label{Basis of subspace}
Let $r \in \mathbb{N}_0$ and $x \in \mathbb{Z}_p / p^r \mathbb{Z}_p$. Then the set $\left\{ \psi_{m,n} : (m,n) \in S \right\}$, where
$$
\psi_{m,n}(z) =
\begin{cases}
p^{r/2} \chi_{m,n}\left( \dfrac{z - x}{p^r} \right), & \text{if } z \in x + p^r \mathbb{Z}_p, \\
0, & \text{otherwise},
\end{cases}
$$
is an orthonormal basis of the subspace $\mathbbm{1}_{x + p^r \mathbb{Z}_p} \cdot L^2\left(\mathbb{Z}_p\right)$. Moreover, for each $(m,n) \in S$, we have
$$\psi_{m,n} = \begin{cases}
p^{-r/2}\displaystyle\sum_{s\leq r} \overline{\chi_{l,s}(x)}\cdot\chi_{l,s}, &\text{ if } (m,n)=(1,0)\\
p^{-r/2}\displaystyle\sum_{l\equiv m \mod p^n}  \overline{\chi_{l,n+r}(x)}\chi_{l,n+r}, &\text{ otherwise}. 
\end{cases}$$
\end{lmma}

\begin{proof}
Let $\iota : \mathbb{Z}_p \to p^r \mathbb{Z}_p$ be the isomorphism given by $\iota(y) = p^r y$ for $y \in \mathbb{Z}_p$. Then the set $\left\{ \chi_{m,n} \circ \iota^{-1} : (m,n) \in S \right\}$ forms an orthogonal basis of $L^2\left(p^r \mathbb{Z}_p\right)$, and hence an orthogonal basis of $L^2\left(p^r\mathbb{Z}_p,\mu|_{p^r\mathbb{Z}_p}\right)$, where each $\chi_{m,n} \circ \iota^{-1}$ is identified as a function on $\mathbb{Z}_p$ by extending it to be zero outside $p^r \mathbb{Z}_p$.

Let $U$ denote the left regular action of $\mathbb{Z}_p$ on $L^2\left(\mathbb{Z}_p\right)$. Then $\left\{ U_x\left( \chi_{m,n} \circ \iota^{-1} \right) : (m,n) \in S \right\}$ is an orthogonal basis of $L^2\left(x + p^r \mathbb{Z}_p,\mu|_{x+p^r\mathbb{Z}_p}\right)$. Since the norm of each such function in $L^2\left(\mathbb{Z}_p\right)$ is $p^{-r/2}$, it follows that $\left\{ p^{r/2} U_x\left( \chi_{m,n} \circ \iota^{-1} \right) : (m,n) \in S \right\}$ is an orthonormal basis of $L^2\left(x + p^r \mathbb{Z}_p,\mu|_{x+p^r\mathbb{Z}_p}\right) = \mathbbm{1}_{x + p^r \mathbb{Z}_p} \cdot L^2\left(\mathbb{Z}_p\right)$.
Hence, the first part of the lemma follows by observing that $\psi_{m,n} = p^{r/2} U_x\left( \chi_{m,n} \circ \iota^{-1} \right)$.

For the second part, let $(l,s)\in S$. Then
$$I:=\left\langle \psi_{m,n}, \chi_{l,s}\right\rangle =\int_{x+p^r\mathbb{Z}_p} \psi_{m,n}(z)\overline{\chi_{l,s}(z)}\,d\mu(z) = p^{r/2}\int_{x+p^r\mathbb{Z}_p}\chi_{m,n}\left(\frac{z-x}{p^r}\right)\cdot\overline{\chi_{l,s}(z)}\,d\mu(z).$$
Applying the transformation $z=x+p^rw$, we get $d\mu (z) = \frac{1}{p^r} d\mu (w)$, and the above integration transforms into
$$p^{-r/2}\int_{\mathbb{Z}_p}\chi_{m,n}(w)\overline{\chi_{l,s}\left(x+p^rw\right)}\,d\mu(w) = p^{-r/2}\overline{\chi_{l,s}(x)}\int_{\mathbb{Z}_p}\chi_{m,n}(w)\overline{\chi_{l,s}(w)}^{p^r}\,d\mu(w).$$
Since 
$$\overline{\chi_{l,s}(w)}^{p^r} = \begin{cases}
\overline{\chi_{1,0}(w)}, &\text{ if } s\leq r,\\
\overline{\chi_{l-\left\lfloor\frac{l}{p^{s-r}}\right\rfloor p^{s-r},s-r}(w)}, &\text{ if } s>r,
\end{cases}$$
we get $I = \begin{cases}
p^{-r/2}\overline{\chi_{l,s}(x)}\delta_{m,1}\cdot\delta_{n,0}, &\text{ if } s\leq r\\
p^{-r/2}\overline{\chi_{l,s}(x)}\delta_{m,l-\left\lfloor\frac{l}{p^{s-r}}\right\rfloor p^{s-r}}\cdot\delta_{n,s-r}, &\text{ if } s>r
\end{cases}$.

Hence $\psi_{1,0} = \displaystyle\sum_{(l,s)\in S} \langle \psi_{1,0},\chi_{l,s}\rangle \chi_{l,s} = p^{-r/2} \sum_{s\leq r} \overline{\chi_{l,s}(x)}\chi_{l,s}$, and for $(m,n)\in S\setminus\{(1,0)\}$,
$$\psi_{m,n} = \sum_{(l,s)\in S} \langle \psi_{m,n},\chi_{l,s}\rangle \chi_{l,s} = p^{-r/2}\sum_{l\equiv m \mod p^n} \overline{\chi_{l,n+r}(x)} \chi_{l,n+r}.$$
\end{proof}

\begin{thm}
Let $\pi$ be the GNS representation of $\mathcal{O}(\mathbb{Z}_p)$ on $L^2(\mathbb{Z}_p)$. Then there does not exist any nontrivial equivariant even spectral triple
$$
\left(\mathcal{O}(\mathbb{Z}_p),\, \pi \oplus \pi,\,
\begin{pmatrix}
0 & D \\
D^* & 0
\end{pmatrix},\,
\begin{pmatrix}
1 & 0 \\
0 & -1
\end{pmatrix}
\right)
$$
of $C(\mathbb{Z}_p)$.
\end{thm}

\begin{proof}
Suppose, to the contrary, that such a nontrivial equivariant even spectral triple exists. By Proposition~\ref{Equivariance}, for each $(m,n)\in S$, there exists $\lambda_{m,n} \in \mathbb{C}$ such that $D\chi_{m,n} = \lambda_{m,n} \chi_{m,n}$. Let us denote the Dirac operator $\begin{pmatrix}
0 & D \\
D^* & 0
\end{pmatrix}$ by $\mathcal{D}$. Then, nontriviality implies the existence of $r \in \mathbb{N}$ and $x \in \mathbb{Z}_p / p^r \mathbb{Z}_p$ such that the index pairing with the projection $q := \mathbbm{1}_{x + p^r \mathbb{Z}_p} \in K_0(C(\mathbb{Z}_p))$ is nonzero, i.e., $\mathrm{Ind}_{\mathcal{D}}(q) := \mathrm{Index}(qDq) \neq 0$. Consider the Fredholm operator $qDq$ on $qL^2(\mathbb{Z}_p)$, and evaluate it on the basis $\psi_{m,n}$, $(m,n) \in S$, as in Lemma~\ref{Basis of subspace}. For $(m,n) \in S \setminus \{(1,0)\}$,
\begin{align*}
qDq(\psi_{m,n}) &= qD(\psi_{m,n}) \\
&= q\left(p^{-r/2} \sum_{l \equiv m \mod p^n} \overline{\chi_{l,n+r}(x)} \lambda_{l,n+r} \cdot\chi_{l,n+r}\right) \tag*{\text{(by Lemma~\ref{Basis of subspace})}} \\
&= \left\langle p^{-r/2} \sum_{l \equiv m \mod p^n} \overline{\chi_{l,n+r}(x)} \lambda_{l,n+r} \cdot \chi_{l,n+r},\, \psi_{m,n} \right\rangle \psi_{m,n} \\
&= \left(p^{-r} \sum_{l \equiv m \mod p^n} \lambda_{l,n+r}\right) \psi_{m,n}.
\end{align*}
Similarly, $qDq(\psi_{1,0}) = \left(p^{-r} \sum_{s \leq r} \lambda_{l,s} \right) \psi_{1,0}$.
Hence $qDq$ is a diagonal operator on $qL^2(\mathbb{Z}_p)$, and therefore $\mathrm{Ind}_{\mathcal{D}}(q) = 0$, a contradiction. This completes the proof.
\end{proof}

Note that the Monna map $T$ (see \ref{Eq 4.1}) provides the parametrization of $\widehat{\mathbb{Z}_p}$ as follows. For $n\in\mathbb{N}_0$, let $\chi_n:\mathbb{Z}_p\to\mathbb{T}$ be the the element of $\widehat{Z}_p$ specified by $\chi_n(1) = e^{2\pi i T(n)}$. Then $\widehat{\mathbb{Z}_p} = \left\{\chi_n:n\in\mathbb{N}_0\right\}$. Let $\sigma : \mathbb{N}_0 \times \mathbb{N}_0\to\mathbb{N}_0$ be the map such that $\chi_m\cdot\chi_n = \chi_{\sigma (m,n)}$, $m,n\in\mathbb{N}_0$. Moreover, for $m,n\in\mathbb{N}_0$, define $\sigma^0\left(m,n\right) =n$, and inductively for each $i\in\mathbb{N}_0$ define 
$$\sigma^{i+1}\left(m,n\right) = \sigma\left(m,\sigma^i\left(m,n\right)\right).$$
We first begin by proving the following lemmas.

\begin{lmma}\label{Value of sigma 1}
Let $m,n\in\mathbb{N}_0$. If there exists $0\leq k\leq m$ such that $n_k\neq p-1$, we denote $\kappa(m,n) = \max\{k: 0\leq k\leq m,\, n_k\neq p-1\}$. Then
$$\sigma\left(p^m,n\right) = \begin{cases}
n+p^{\kappa (m,n)} + p^{\kappa (m,n) +1} - p^{m+1}, &\text{if } \exists 0\leq k\leq m \text{ such that } n_k\neq p-1,\cr
n+1-p^{m+1}, &\text{otherwise}. \cr
\end{cases}$$
\end{lmma}

\begin{proof}
We have 
\begin{align*}
& T\left(p^m\right) + T(n)\\
&=\frac{1}{p^{m+1}} + \sum_{k=0}^\infty \frac{n_k}{p^{k+1}}\\
&=\begin{cases}
\sum_{k=0}^{\kappa(m,n)-1}\frac{n_k}{p^{k+1}} + \frac{n_{\kappa(m,n)} +1}{p^{\kappa (m,n)}+1} + \sum_{k=m+1}^{\infty} \frac{n_k}{p^{k+1}}, &\text{if } \exists 0\leq k\leq m \text{ such that } n_k\neq p-1\\
1+\sum_{k=m+1}^{\infty}\frac{n_k}{p^{k+1}}, &\text{otherwise}
\end{cases}\\
&=\begin{cases}
T\left(n+p^{\kappa(m,n)} + p^{\kappa(m,n)+1} - p^{m+1}\right), &\text{if } \exists 0\leq k\leq m \text{ such that } n_k\neq p-1, \cr
1+ T\left(n+1-p^{m+1}\right), &\text{otherwise}.\cr
\end{cases}
\end{align*}
Therefore, we have
\begin{align*}
e^{2\pi i T\left(\sigma\left(p^m,n\right)\right)} &= \chi_{\sigma\left(p^m,n\right)}(1)\\
&= \chi_{p^m}\cdot\chi_n (1)\\
&= e^{2\pi i\left(T\left(p^m\right) + T(n)\right)}\\
&=\begin{cases}
e^{2\pi i T\left(n+p^{\kappa(m,n)} + p^{\kappa(m,n)+1} - p^{m+1}\right)}, &\text{if } \exists 0\leq k\leq m \text{ such that } n_k\neq p-1\\
e^{2\pi iT\left(n+1-p^{m+1}\right)}, &\text{otherwise}.
\end{cases}
\end{align*}
Since $T\left[\mathbb{N}_0\right]\subseteq [0,1)$, and $T|_{\mathbb{N}_0}$ is injective,  we get the claim.
\end{proof}

\begin{lmma}\label{Characterization of phi}
Let $m,l\in\mathbb{N}_0$. Then
$$\sigma^i\left(p^m,lp^{m+1}\right) = \begin{cases}
lp^{m+1} + \sum_{k=0}^m i_{m-k} p^k, &\text{if } 0\leq i\leq p^{m+1}-1\\
\sigma^{i-\left\lfloor\frac{i}{p^{m+1}}\right\rfloor p^{m+1}}\left(p^m,lp^{m+1}\right), &\text{if } i\geq p^{m+1}.
\end{cases}$$
\end{lmma}

\begin{proof}
We first prove the given expression of $\sigma^i\left(p^m,lp^{m+1}\right)$ when $0\leq i\leq p^{m+1}-1$, by applying induction on the variable $i$. Note that by definition the expression holds for $i=0$. Let $0\leq i\leq p^{m+1}-2$, and assume that the expression of $\sigma^i\left(p^m,lp^{m+1}\right)$ holds true. Define $\kappa_0(i) =\min\{k: 0\leq k\leq m,\, i_k\neq p-1\}$. Observe that $\kappa\left(m,\sum_{k=0}^m i_{m-k} p^k\right) = m-\kappa_0(i)$, and 
$$(i+1)_k =\begin{cases}
0, &\text{if } 0\leq k <\kappa_0(i)\\
i_{\kappa_0(i)} +1, &\text{if } k=\kappa_0(i)\\
i_k, &\text{if } \kappa_0(i)<k\leq m.
\end{cases}$$ 
Let $I=\{0,1,\ldots,m\}$. We then have
\begin{align*}
&\sigma^{i+1}\left(p^m,lp^{m+1}\right)\\
&= \sigma\left(p^m,\sigma^i\left(p^m,lp^{m+1}\right)\right)\\
&= \sigma\left(p^m,lp^{m+1} + \sum_{k=0}^m i_{m-k} p^k\right)\\
&= lp^{m+1} + \sum_{k=0}^m i_{m-k} p^k + p^{m-\kappa_0(i)} + p^{m-\kappa_0(i)+1} - p^{m+1}\\
&= lp^{m+1} + \left(\sum_{\substack{k\in I\\ k\leq m-\kappa_0(i)}} i_{m-k} p^k\right) + \left(\sum_{\substack{k\in I\\ k>m-\kappa_0(i)}} (p-1)p^k\right) + p^{m-\kappa_0(i)} + p^{m-\kappa_0(i) +1} - p^{m+1}\\
&= lp^{m+1} + \left(\sum_{\substack{k\in I\\ k< m-\kappa_0(i)}} i_{m-k} p^k + i_{\kappa_0(i)} p^{m-\kappa_0(i)}\right) + \left(p^{m+1} - p^{m-\kappa_0(i) +1}\right) + p^{m-\kappa_0(i)} + p^{m-\kappa_0(i) +1} - p^{m+1}\\
&= lp^{m+1} + \sum_{\substack{k\in I\\ k< m-\kappa_0(i)}} i_{m-k} p^k + \left(i_{\kappa_0(i)} +1\right)p^{m-\kappa_0(i)}\\
&= lp^{m+1} + \sum_{\substack{k\in I\\ k< m-\kappa_0(i)}} (i+1)_{m-k} p^k + (i+1)_{\kappa_0(i)} p^{m-\kappa_0(i)} + \left(\sum_{\substack{k\in I\\ k> m-\kappa_0(i)}} (i+1)_{m-k} p^k\right)\\
&= lp^{m+1} + \sum_{k=0}^m (i+1)_{m-k} p^k,
\end{align*}
and hence the expression of $\sigma^i\left(p^m,0\right)$, for $0\leq i\leq p^{m+1}-1$, get established by induction on $i$, as promised. Moreover, we have
\begin{align*}
\sigma^{p^{m+1}}\left(p^m,lp^{m+1}\right) &= \sigma\left(p^m,\sigma^{p^{m+1}-1}\left(p^m,lp^{m+1}\right)\right)\\
&=\sigma\left(p^m,lp^{m+1} + \sum_{k=0}^m (p-1)p^k\right)\\
&=lp^{m+1} + \sum_{k=0}^m (p-1)p^k +1 - p^{m+1}\quad (\text{by Lemma } \ref{Value of sigma 1})\\
&= lp^{m+1} = \sigma^0\left(p^m,lp^{m+1}\right),
\end{align*}
and therefore we get for each $i\geq p^{m+1}$ that $\sigma^i\left(p^m,lp^{m+1}\right) = \sigma^{i-\left\lfloor \frac{i}{p^{m+1}}\right\rfloor p^{m+1}} \left(p^m,lp^{m+1}\right)$. This completes the proof.
\end{proof}

Let $m \in \mathbb{N}_0$. For each $l \in \mathbb{N}_0$, define $P_{m,l} = \{ n + l p^{m+1} : n = 0,1,\dots, p^{m+1}-1 \}$. Obseve the collection $P = \{ P_{m,l} : l \in \mathbb{N}_0 \}$ partitions $\mathbb{N}_0$. Also, note that for fixed $l,c\in\mathbb{N}_0$ with $0\leq c<p^{m+1}$, one has $$\left\{\sigma^i\left(p^m,lp^{m+1}+c\right):0\leq i<p^{m+1}\right\} = \left\{lp^{m+1} + d : 0\leq d<p^{m+1}\right\}.$$

\begin{lmma}\label{partition}
Let $\phi: \mathbb{N}_0 \to \mathbb{N}_0$ be an injection such that the set $\mathbb{N}_0 \setminus \phi(\mathbb{N}_0)$ is nonempty and finite. Then, there exists an $M \in \mathbb{N}$ such that for all $m \geq M$ and any $L \in \mathbb{N}$, there exist $l_1, l_2, l_3 \in \mathbb{N}_0$ with $l_1 \geq L$ and $l_2 \neq l_3$ satisfying $\phi(P_{m,l_1}) \cap P_{m,l_2} \neq \emptyset$ and $\phi(P_{m,l_1}) \cap P_{m,l_3} \neq \emptyset$.
\end{lmma} 

\begin{proof}
Choose $M \in \mathbb{N}$ such that $p^{M+1} > |\mathbb{N}_0 \setminus \phi(\mathbb{N}_0)|$, and fix $m \geq M$. Let $L \in \mathbb{N}$ and consider the sets $P_{m,l}$ for $l \geq L$. Assume, for contradiction, that for each $l \geq L$, there exists $l' \geq 0$ such that $\phi(P_{m,l}) = P_{m,l'}$. Then we have

\[
\left|\mathbb{N}_0 \setminus \phi\left(\bigcup_{l=L}^{\infty} P_{m,l}\right)\right| = k p^{m+1}
\]
for some $k \in \mathbb{N}$ or $k = \infty$. Since $\mathbb{N}_0 \setminus \phi(\mathbb{N}_0)$ is finite, the case $k = \infty$ is impossible. Moreover, injectivity of $\phi$ implies  
\[
\left|\phi\left(\bigcup_{l=0}^{L-1} P_{m,l}\right)\right| = L p^{m+1}.
\]
Thus, for a finite $k$, we obtain 
\[
0 < \left|\mathbb{N}_0 \setminus \phi(\mathbb{N}_0)\right| =
\left|\left(\mathbb{N}_0 \setminus \phi\left(\bigcup_{l=L}^{\infty} P_{m,l}\right)\right) \setminus \bigcup_{l=0}^{L-1} P_{m,l} \right| = (k - L)p^{m+1},
\]
contradicting the bound $\left|\mathbb{N}_0 \setminus \phi(\mathbb{N}_0)\right| < p^{m+1}$. Hence there exists some $l_1 \geq L$ such that $\phi(P_{m,l_1})$ intersects at least two distinct sets of the form $P_{m,l}$, completing the proof.
\end{proof}

\begin{lmma}\label{Nonexistence of phi}
There does not exist an injective function $\phi :\mathbb{N}_0\to\mathbb{N}_0$ such that the following two conditions hold.
\begin{enumerate}[(i)]
\item\label{cond1, nonexistence} The set $\mathbb{N}_0\setminus \phi\left(\mathbb{N}_0\right)$ is a nonempty finite set.
\item\label{cond2, nonexistence} For each $m\in\mathbb{N}_0$, there exists $m_0\in\mathbb{N}_0$ such that $\phi\left(\sigma\left(p^m,n\right)\right) = \sigma\left(p^m,\phi(n)\right)$ for each $n\geq m_0$.
\end{enumerate}
\end{lmma}

\begin{proof}
Suppose that there exists an injective function $\phi : \mathbb{N}_0\to\mathbb{N}_0$ satisfying  (\ref{cond1, nonexistence}) and (\ref{cond2, nonexistence}).  Choose $M$ such that $p^{M+1} > \left|\mathbb{N}_0\setminus\phi\left(\mathbb{N}_0\right)\right|$, and let $m\geq M$. By condition (\ref{cond2, nonexistence}), there exists $m_0\in\mathbb{N}_0$ such that  
$$\phi\left(\sigma\left(p^m,n\right)\right) = \sigma\left(p^m,\phi(n)\right) \quad \text{for all } n\geq m_0.$$  
Pick $L\in\mathbb{N}$ such that $Lp^{m+1}\geq m_0$. By Lemma \ref{partition}, one can take $l_1, l_2, l_3\in\mathbb{N}_0$ with $l_1\geq L$ and $l_2\neq l_3$ such that  
$$\phi\left(P_{m,l_1}\right)\cap P_{m,l_2} \neq \emptyset \quad \text{and} \quad \phi\left(P_{m,l_1}\right)\cap P_{m,l_3} \neq \emptyset.$$ 
Let $x,y\in P_{m,l_1}$ be such that $\phi(x)\in P_{m,l_2}$ and $\phi(y)\in P_{m,l_3}$. By Lemma \ref{Characterization of phi}, there exist $i_1,i_2$ with $0\leq i_1, i_2 < p^{m+1}$ such that  
$$x = \sigma^{i_1}\left(p^m,l_1 p^{m+1}\right) \quad \text{and}\quad  y = \sigma^{i_2}\left(p^m,l_1 p^{m+1}\right).$$ 
Since $x\neq y$, we have $i_1\neq i_2$, i.e., either $i_1<i_2$ or $i_1> i_2$.
\begin{itemize}
\item \textbf{Case 1:} If $i_1 < i_2$, define 
$$i_3 = \min\left\{ i : i_1\leq i<p^{m+1}-1,\,\phi\left(\sigma^i\left(p^m,l_1p^{m+1}\right)\right)\in P_{m,l_2},\,\phi\left(\sigma^{i+1}\left(p^m,l_1p^{m+1}\right)\right)\notin P_{m,l_2} \right\}.$$ 
Since $i_1 < i_2$, $\phi\left(\sigma^{i_1}\left(p^m,l_1 p^{m+1}\right)\right)\in P_{m,l_2}$ and $\phi\left(\sigma^{i_2}\left(p^m,l_1p^{m+1}\right)\right)\notin P_{m,l_2}$, we have that $i_3$ exists with $i_1\leq i_3\leq i_2-1$. Since $\sigma^{i_3}\left(p^m, l_1p^{m+1}\right) \in P_{m,l_1}$, by condition (\ref{cond2, nonexistence}),  
$$\phi\left(\sigma^{i_3 +1}\left(p^m,l_1 p^{m+1}\right)\right) = \phi\left(\sigma\left(p^m,\sigma^{i_3}\left(p^m,l_1 p^{m+1}\right)\right)\right) = \sigma\left(p^m,\phi\left(\sigma^{i_3}\left(p^m,l_1p^{m+1}\right)\right)\right)\in P_{m,l_2},$$
which contradicts the definition of $i_3$.
\item \textbf{Case 2:} If $i_1 > i_2$, following an analogous argument as in Case 1, we again arrive at  a contradiction.
\end{itemize}
Thus, no such injective function $\phi$  exists.
\end{proof}

\begin{thm}\label{main}
There does not exist an even spectral triple  
$$\left(\mathcal{O}(\mathbb{Z}_p), \pi\oplus\pi,  
\begin{pmatrix}
0 & D \\
D^* & 0
\end{pmatrix}, \begin{pmatrix}
1 & 0 \\
0& -1
\end{pmatrix}
\right)$$
of  $C(\mathbb{Z}_p)$, where $\pi$ is the GNS representation of $\mathcal{O}\left(\mathbb{Z}_p\right)$ on $L^2\left(\mathbb{Z}_p\right)$, and $D$ is an injective operator with $\mathrm{Index}(D)\neq 0$ satisfying  
$$D(\chi_n) = \lambda_n \chi_{\phi(n)},\, \lambda_n\in\mathbb{C},\, n\in\mathbb{N}_0,$$
for some function $\phi: \mathbb{N}_0 \to \mathbb{N}_0$.
\end{thm}

\begin{proof}
Suppose that such a spectral triple exists. Since $D$ is injective, $\phi$ must also be injective, and its nonzero index implies that $\mathbb{N}_0\setminus\phi(\mathbb{N}_0)$ is a nonempty finite set.
For each $m \in \mathbb{N}_0$, the commutator $\left[D, \pi(\chi_{p^m})\right]$ is bounded. Note that, for all $m,n\in\mathbb{N}_0$,
\[
\left[D, \pi(\chi_{p^m})\right]\chi_n = \lambda_{\sigma(p^m,n)}\chi_{\phi(\sigma(p^m,n))} - \lambda_n\chi_{\sigma(p^m,\phi(n))}.
\]
Define the set $S_m = \{n \in \mathbb{N}_0 : \phi(\sigma(p^m,n)) \neq \sigma(p^m,\phi(n))\}$. For each $n\in S_m$, we obtain the bound
\[
\|\left[D, \pi(\chi_{p^m})\right]\chi_n\| = \left(\left|\lambda_{\sigma(p^m,n)}\right|^2 + \left|\lambda_n\right|^2\right)^{1/2} \leq \|\left[D, \pi(\chi_{p^m})\right]\|.
\]
Since $D$ has compact resolvent, we have $\lambda_n \to \infty$ as $n\to\infty$, which forces $S_m$ to be finite. Setting $m_0 = \max S_m +1$, it follows that $\phi(\sigma(p^m,n)) = \sigma(p^m,\phi(n))$, for all $n \geq m_0$.
However, by Lemma \ref{Nonexistence of phi}, no such function $\phi$  exist, leading to a contradiction. Hence such spectral triple does not exist.
\end{proof}

\begin{rmrk}
It is natural to ask whether any equivariant even spectral triple, or any such spectral triple with a weighted shift-type Dirac operator, fails to exist for the $p$-adic Heisenberg groups $\mathbb{H}_d(\mathbb{Z}_p)$, as shown in the case of $\mathbb{Z}_p$. Determining the existence or nonexistence of nontrivial equivariant even spectral triples for such compact Vilenkin groups remains an open and interesting problem. We plan to investigate these questions in a subsequent article.
\end{rmrk}

\bigskip

\noindent\begin{footnotesize}\textbf{Acknowledgement}:
The first named author sincerely thanks Professor Partha Sarathi Chakraborty for his valuable suggestions and insightful discussions. A part of this work was carried out during his tenure as a postdoctoral fellow at IIT Gandhinagar, and he gratefully acknowledges all the support provided by the institute. He also acknowledges support from the INSPIRE Faculty Fellowship research grant DST/INSPIRE/04/2021/002620, DST, Government of India. Both authors acknowledge support from the NBHM research project grant\linebreak 02011/19/2021/NBHM(R.P)/R\&D II/8758.
\end{footnotesize}

\bigskip

\noindent{\sc Surajit Biswas}  (\texttt{surajitb@iiitd.ac.in},
\texttt{241surajit@gmail.com})\\
    {\footnotesize Department of Mathematics,\\ Indraprastha Institute of Information Technology Delhi,\\ Okhla Industrial Estate, Phase III, New Delhi, Delhi 110020, India}\\

\noindent{\sc Bipul Saurabh} (\texttt{bipul.saurabh@iitgn.ac.in},  \texttt{saurabhbipul2@gmail.com})\\
	{\footnotesize Department of Mathematics,\\ Indian Institute of Technology, Gandhinagar,\\  Palaj, Gandhinagar 382055, India}

\end{document}